\documentclass[11pt]{amsart}

\usepackage[utf8]{inputenc}
\usepackage[T1]{fontenc}
\usepackage[english]{babel}
\usepackage{mathtools}

\usepackage{amsmath, amsfonts, amssymb, amsthm, mathtools, bbm}
\usepackage{tensor}

\usepackage[mathscr]{eucal}
\usepackage{relsize}

\usepackage{graphicx}
\usepackage{xcolor}
\usepackage[a4paper, twoside=false, vmargin={2cm,2.5cm}, includehead]{geometry}
\usepackage{tikz-cd}

\usepackage{enumitem}
\setlist{nolistsep}

\usepackage[linktocpage=true]{hyperref}
\usepackage[capitalize]{cleveref}

\usepackage{pdfpages}
\usepackage[official]{eurosym}
\usepackage[stable]{footmisc}
\usepackage{comment}
\usepackage{cancel}
\usepackage{pgfgantt}
\usepackage{fancyhdr}

\newtheoremstyle{plain}{3mm}{3mm}{\slshape}{}{\bfseries}{.}{.5em}{}
\newtheoremstyle{definition}{2mm}{2mm}{}{}{\bfseries}{.}{.5em}{}

\theoremstyle{plain}

\newtheorem{Theorem}{Theorem}[section]
\newtheorem{Lemma}[Theorem]{Lemma}
\newtheorem{Proposition}[Theorem]{Proposition}
\newtheorem{Corollary}[Theorem]{Corollary}

\newtheorem{Conjecture}[Theorem]{Conjecture}
\newtheorem{Problem}[Theorem]{Problem}

\theoremstyle{definition}
\newtheorem{Definition}[Theorem]{Definition}
\newtheorem{Remark}[Theorem]{Remark}

\newtheorem{Claim}[Theorem]{Claim}

\newtheorem{innercustomthm}{Theorem}
\newenvironment{customthm}[1]
{\renewcommand\theinnercustomthm{#1}\innercustomthm}
{\endinnercustomthm}

\numberwithin{equation}{section}

\ganttset{
  group/.append style={orange},
  milestone/.append style={red},
  progress label node anchor/.append style={text=red}
}

\newcommand{\Oh}{{\rm O}}
\newcommand{\oh}{{\rm o}}
\newcommand{\e}{\varepsilon}

\newcommand{\norm}[1]{\left\Vert #1 \right\Vert}
\newcommand{\nnorm}[1]{\lvert\!|\!| #1 |\!|\!\rvert}

\newcommand{\N}{\mathbb{N}}
\newcommand{\Z}{\mathbb{Z}}
\newcommand{\Q}{\mathbb{Q}}
\newcommand{\R}{\mathbb{R}}
\newcommand{\C}{\mathbb{C}}
\newcommand{\T}{\mathbb{T}}

\newcommand{\I}{\mathcal{I}}

\newcommand{\X}{\mathcal{X}}

\newcommand{\CC}{\mathcal{C}}

\newcommand{\CF}{\mathcal F}
\newcommand{\CG}{\mathcal G}

\newcommand{\CL}{\mathcal L}

\newcommand{\CP}{\mathcal P}
\newcommand{\CQ}{\mathcal Q}
\newcommand{\CR}{\mathcal R}
\newcommand{\CS}{\mathcal S}

\newcommand{\CY}{\mathcal Y}
\newcommand{\CZ}{\mathcal Z}

\newcommand{\Hilb}{\mathscr{H}}

\newcommand{\1}{\mathbbm{1}}

\newcommand{\xm}{(X,\mu)}

\newcommand{\xms}{(X,\mu,S)}

\newcommand{\E}{\raisebox{-0.4ex}{\scalebox{1.4}[1.3]{$\mathbb{E}$}}}
\newcommand{\logE}{\E^{\log}}

\renewcommand{\Re}{\mathrm{Re}}

\renewcommand{\P}{\mathbb{P}}

\renewcommand{\epsilon}{\varepsilon}

\newcommand{\Loc}{\textup{Loc}}

\newcommand{\wt}{\widetilde}

\def \colon{{:}\;}

\newcommand{\floor}[1]{{\left \lfloor #1 \right \rfloor}}

\newcommand{\Cesaro}{Ces\`{a}ro}

\newcommand{\Folner}{F\o{}lner}

\newcommand{\Szemeredi}{Szemer\'{e}di}

\newcommand{\Katai}{K\'{a}tai}
\newcommand{\Mobius}{M\"{o}bius}

\newcommand{\Turan}{Tur{\'a}n}

\newcommand{\Halasz}{Hal\'{a}sz}

\usepackage[normalem]{ulem}

\title{Multiple ergodic averages for commuting multiplicative actions}

\author{Dimitrios Charamaras}
\address[Dimitrios Charamaras]{Centre de recherches math\'ematiques\\
Universit\'e de Montr\'eal, Montr\'eal, QC H3T 1J4, Canada}
\email{dimitrios.charamaras@umontreal.ca}

\author{Andreas Koutsogiannis}
\address[Andreas Koutsogiannis]{
Department of Mathematics, Aristotle University of Thessaloniki, Thessaloniki 54124, Greece}
\email{akoutsogiannis@math.auth.gr}

\author{Konstantinos Tsinas}
\address[Konstantinos Tsinas]{Institute of Mathematics, \'{E}cole Polytechnique F\'{e}d\'{e}rale de Lausanne (EPFL), Lausanne 1015,
Switzerland}
\email{konstantinos.tsinas@epfl.ch}

\date{}

\thanks{This work was initiated while the first author was supported by the Swiss National Science Foundation grant TMSGI2-211214, and it was completed while he was supported by the CRM-ISM Postdoctoral Fellowship and the Fondation Courtois through the Chaire Courtois en recherche fondamentale - II.
For the second author, the research project is implemented in the framework of H.F.R.I call ``3rd Call for H.F.R.I.’s
Research Projects to Support Faculty Members \& Researchers'' (H.F.R.I. Project Number: 24979).  
The third author was supported by the Swiss National Science Foundation grant TMSGI2-211214.
For the purpose of Open Access, the authors have applied a CC-BY public copyright license to any Author Accepted Manuscript (AAM) version arising from this submission.}



\begin{document}

\begin{abstract} 
We study the convergence of multiple ergodic averages involving several commuting multiplicative actions. We prove that for finitely generated systems, the averages 
$$\frac{1}{N}\sum_{n=1}^N S_{1,n}F_1\cdot\ldots\cdot  S_{\ell,n}F_\ell$$
converge in norm, settling a conjecture of Frantzikinakis.
 

Our proofs feature tools from ergodic theory and probabilistic number theory. We exploit some stability properties of additive functions to establish a novel variant of \Katai's orthogonality criterion, which leads to box seminorm control for our averages.  Using the theory of magic extensions introduced by Host and a delicate complexity reduction argument, we eventually reduce our problem to known convergence results.
\end{abstract}

\maketitle


\section{Introduction}\label{section_intro}

\subsection{History}

The study of the limiting behavior of multiple ergodic averages has been a major field of study originating from Furstenberg's influential work on Szemer\'{e}di's theorem. In \cite{Furstenberg-original}, Furstenberg derived some information on the limiting behavior of the averages \begin{equation}\label{eq_Furstenberg averages}
    \frac{1}{N}\sum_{n=1}^N \int_X F_0\cdot T^nF_1\cdot T^{2n}F_2\cdot\ldots\cdot T^{\ell n}F_\ell\,d\mu,
\end{equation} 
where $\ell$ is a positive integer, $T$ is an invertible measure-preserving map acting on a probability space $(X,\X,\mu)$ and $F_0,\ldots, F_\ell$ are 1-bounded measurable functions. Note that the transformation induces
an action of the group $(\Z,+)$ on $X$ by $(n,x)\mapsto T^nx$.

Restricting to $F_0=\ldots= F_\ell=\1_{A}$ for some $A\in \X$ with $\mu(A)>0$, Furstenberg's multiple recurrence result is the assertion that the limit inferior of the averages in \eqref{eq_Furstenberg averages} is positive, a fact that can be shown to be equivalent to \Szemeredi's theorem on the existence of arbitrarily long arithmetic progressions in large subsets of integers.

Strictly speaking, the methods introduced in \cite{Furstenberg-original} could not establish whether the averages
\begin{equation}\label{eq_Furstenberg averages 2}
     \frac{1}{N}\sum_{n=1}^N  T^nF_1\cdot T^{2n}F_2\cdot\ldots\cdot T^{\ell n}F_\ell
\end{equation}
converge in $L^2(\mu)$, at least for $\ell\geq 3$.\footnote{The pointwise convergence of these averages is a notorious open problem, only known in the case $\ell=2$ by work of Bourgain \cite{Bourgain_pointwise}.} The mean convergence of the averages was a major open problem until 2005, when it was resolved by Host and Kra \cite{Host-Kra-annals} (building on work of Conze and Lesigne \cite{Conze_Lesigne}, who verified the conjecture for $\ell=3$ and totally ergodic systems) with a different proof by Ziegler \cite{Ziegler_jams}.
The highlight of the newly introduced tools was not only the proof of the convergence, but the introduction of some factors of the system $(X,\mu,T)$ that
control the behavior of the averages \eqref{eq_Furstenberg averages 2} and which have an explicit algebraic structure. This allows the calculation of the limit (carried out in \cite{Ziegler_jams}) of the averages in \eqref{eq_Furstenberg averages 2}. Moreover, the methods extend to averages for $\Z$-actions with iterates involving more complicated sequences, such as polynomials \cite{Leibman-several}, functions arising from Hardy fields \cite{Fra-Hardy-singlecase}
and more (see \cite{Fra-open} for the history of these problems).

Multidimensional analogues of \Szemeredi's theorem and other similar combinatorial problems necessitate the study of multiple ergodic averages involving several commuting transformations. For instance, a natural question is whether or not the multiple ergodic averages 
\begin{equation}\label{E:T1}
    \frac{1}{N}\sum_{n=1}^N T_1^n F_1\cdot\ldots\cdot T_\ell^n F_\ell
\end{equation}
converge, where $T_1,\ldots, T_\ell$ are (invertible) commuting (i.e., $T_iT_j=T_jT_i$ for all $i,j\in [\ell]$) measure preserving transformations acting on a probability space $(X,\X,\mu)$. Note that such a collection of transformations $T_1,\ldots, T_\ell$ induces a $(\Z^\ell,+)$-action on $X$.
Here, the situation is radically different. The Host-Kra factors cannot control the behavior of these averages. In fact, simple examples show that any factor controlling the behavior of an average as in \eqref{E:T1} could contain any possible measure-preserving system.  
Nonetheless, the norm-convergence problem was fully resolved by Tao \cite{Tao-L^2commuting} using a transference argument to finitistic objects and then establishing a finite convergence result. This proof was later translated to ergodic theoretic language by Towsner \cite{Towsner} using non-standard methods. Later, Austin \cite{Austin_normconvergence} gave a genuinely ergodic theoretic proof, followed by Host \cite{Host-L^2-commuting-linearcase}, who gave another proof inspired by Austin's arguments. A common idea in these last two articles was the construction of an extension of the original system, wherein certain factors that control the averages admit a simple description that reduces the complexity of the problem. Since the averages converge for all functions in the larger system, the convergence result in the initial system follows.
The drawback of all these approaches is that they provide no description of the limit, unlike the case of a single transformation.
Still, these methods provide a simpler proof of the convergence of the averages \eqref{eq_Furstenberg averages 2}.

\subsection{Multiple ergodic averages in the multiplicative setting}

We discussed the case of $\Z^k$-actions thus far, though all questions can be recast in the setting of a general abelian group (or semigroup) action. Of particular interest is the case of a semigroup action by $(\N,\times)$. A \emph{multiplicative system} is a quadruple $(X,\X,\mu,S)$,
where $(X,\X,\mu)$ is a probability space and $S=(S_n)_{n\in\N}$ is a sequence of measure-preserving transformations satisfying the relation 
\begin{equation*}
    S_{nm}=S_n\circ S_m
\end{equation*}
for any $m,n\in \N$. Note that the previous relation implies that the maps $S_n$ are pairwise commuting. If we assume that the transformations $S_n$ are invertible, then they naturally induce a  $(\Q_{+},\times)$-action on $(X,\X,\mu)$ by $S_{n/m}=S_n\circ S_m^{-1}$.

Simple examples of multiplicative systems include:

i) \emph{rotations by multiplicative functions}: if $f\colon\N\to \mathbb{S}^1$ is a completely multiplicative function, then we define the map $S_n:\mathbb{S}^1\to \mathbb{S}^1$ by $S_nz=f(n)z$, which preserves the Lebesgue measure on $\mathbb{S}^1$.

ii) \emph{$(\Z,+)$ actions evaluated at completely additive sequences}: if $(X,\X,\mu,T)$ is a $(\Z,+)$-system and $a:\N\to \N_0$ is a completely additive function, then the maps $S_n=T^{a(n)}$ yield a multiplicative action on $(X,\X,\mu)$.

iii) \emph{dilations on the torus}: Given a positive integer $k$, the maps $S_nz=z^{n^k}$ preserve the Lebesgue measure and form a multiplicative 
action.

iv) any sequence of commuting measure-preserving transformations $(T_p)_{p\in\P}$ gives a multiplicative action by defining $S_{p^k}=\underbrace{T_p\circ\ldots\circ T_p}_{\text{$k$ terms}}$
and $S_{p_1^{a_1}\ldots p_k^{a_k}}=S_{p_1^{a_1}}\circ\ldots\circ S_{p_k^{a_k}}$.

Multiplicative systems have recently gained prominence due to their connections to partition regularity results for homogeneous quadratic forms \cite{Fra-Klu-Mor}, \cite{fra-klu-mor-2}, \cite{Fra_survey_mult}, \cite{Fra_Mount}. They arise naturally in combinatorial problems that involve sets that are multiplicatively large. Recurrence results for multiplicative actions have been studied \cite{Moreira&friends} (see also \cite{CMT25} and \cite{tafula_sunkai}), though results in this field have been sparse.

We remark that the natural way to study multiplicative actions
is through multiplicative \Folner\ sequences. Namely, these are sequences $\Phi_N\subseteq \N$ that are approximately dilation invariant, in the sense that \begin{equation*}
    \lim_{N\to\infty} \frac{|\Phi_N\cap \Phi_N/a|}{|\Phi_N|}=1
\end{equation*}for every $a\in \N$. Then, we define the ergodic averages \begin{equation*}
    \frac{1}{|\Phi_N|}\sum_{n\in \Phi_N} S_nF
\end{equation*}which converge by the mean ergodic theorem for amenable group actions (\cite[Theorem~3.33]{Glasner_book_2003}). Using a multiplicative averaging scheme, some results from the additive setting admit analogs in the multiplicative setting. Still, the structure theory of Host-Kra factors has not been adapted to multiplicative actions (it is restricted in the general case to finitely generated abelian group actions \cite{Griesmer-thesis} or to lower-order cases \cite{Asgar&friends}) and most results involving multiple averages remain unknown.

Nonetheless, it is an interesting question whether multiplicative actions behave nicely under standard Ces\`{a}ro averages. For example, all results in \cite{Fra-Klu-Mor} or \cite{Moreira&friends} involve additive averages of multiplicative objects.
The simplest possible question that can be formulated in this context is whether the averages 
\begin{equation*}
    \frac{1}{N}\sum_{n\leq N} S_nF
\end{equation*}
converge in any multiplicative system $(X,\X,\mu,S)$ and for any $F\in L^{\infty}(\mu)$.

Specializing to the case of rotations by multiplicative functions, the problem reduces to the study of mean values of 1-bounded completely multiplicative functions \begin{equation*}
    \frac{1}{N}\sum_{n\leq N}f(n).
\end{equation*} This question was resolved in full generality by \Halasz{} \cite{Halasz} (the case of real valued multiplicative functions was solved by earlier work of Wirsing \cite{Wirsing1, Wirsing2}). The case where $f$ is the Liouville function, where the mean value is zero, is equivalent to the prime number theorem.
However, not all completely multiplicative functions have a mean value, and this readily implies that the averages do not necessarily converge in mean in a general multiplicative system.
For instance, if we consider the rotation by the multiplicative function $f(n)=n^i$ and the function $F(z)=z, z\in \mathbb{S}^1$, we deduce that 
\begin{equation*}
      \frac{1}{N}\sum_{n\leq N} S_nF(z)=  \frac{1}{N}\sum_{n\leq N} n^iz=\frac{N^i}{1+i}z+\oh_{N\to\infty}(1),
\end{equation*}
which does not converge in norm.

In \cite{Charamaras-multiplicative}, the first author overcame this problem by restricting attention to a special class of systems. A multiplicative system $(X,\X,\mu, S)$ is called \emph{finitely generated} if the set $\{S_p\colon p\in \P\}$ is finite.
\begin{customthm}{A}[{\cite[Corollary 1.9]{Charamaras-multiplicative}}]
\label{thm_PMET}
    Let $(X,\X,\mu,S)$ be a finitely generated multiplicative system. Then, for any $F\in L^2(\mu)$, the averages \begin{equation*}
        \frac{1}{N}\sum_{n= 1}^N S_nF
    \end{equation*}
    converge in $L^2(\mu)$.
\end{customthm}
Theorem A is a consequence of the Pretentious Mean Ergodic Theorem \cite[Theorem A]{Charamaras-multiplicative}, an analogue of the von Neumann mean ergodic theorem for finitely generated multiplicative systems, which also gives an explicit description of the limit.

Our main result is an extension of Theorem~\ref{thm_PMET} to the setting of multiple ergodic averages. Here, we say that the tuple $(X,\X,\mu,S_1,\ldots,S_\ell)$ is a \emph{(finitely generated) multiplicative system} if each $(X,\X,\mu,S_i),$ $i\in[\ell]$ is a (finitely generated) multiplicative system. The system is \emph{commuting} if $S_{i,n}S_{j,m}=S_{j,m}S_{i,n}$ for all $n,m\in\N, i,j \in [\ell].$

\begin{Theorem}\label{main theorem}
    Let $(X,\X,\mu,S_1,\ldots,S_\ell)$ be a commuting finitely generated multiplicative system. Then, for any functions $F_1,\ldots,F_\ell\in L^\infty(\mu)$, the averages
    \begin{equation}\label{main average general}
       \frac{1}{N}\sum_{n=1}^N S_{1,n}F_1\cdot\ldots\cdot S_{\ell,n}F_\ell
    \end{equation}
    converge in $L^2(\mu)$.
\end{Theorem}

\begin{Remark}\label{remark_invertible extensions}
   We may assume henceforth that the actions $S_{i}$ in the statement of Theorem~\ref{main theorem} are invertible. Indeed, for any system $(X,\X,\mu, S_1,\ldots, S_{\ell})$, we can construct an extension $(\wt{X},\wt{\X}, \wt{\mu}, \wt{S}_1,\ldots, \wt{S}_{\ell})$ (see the following section for definition), where the actions $\wt{S_i}$ are invertible. This can be proven by a standard modification of the invertible extension construction for $(\N,+)$ actions (see the proof of \cite[Theorem 2.6]{Furstenberg-book}).
\end{Remark}

For brevity, we refer to commuting, invertible, finitely generated multiplicative systems simply as \emph{systems}.

Similar to the setting of $\Z$-actions, we note that, in the general case (not necessarily finitely generated multiplicative actions), some commutativity assumption on the actions $S_{i,n}$ is needed. While analogous results might be true if we assume that the actions generate a nilpotent group, it does not hold for arbitrary collections of actions. To see this, one can modify a construction of Berend \cite{Berend85}.\footnote{On the sequence space $\{0,1\}^{\Q_{+}}$ with the Bernoulli $(1/2,1/2)$-measure, we define the action $T_n$ by $(T_nx)_q=x_{qn},$  we consider a permutation $\pi$ of $\Q_{+}$ such that $\pi(q)\in \{q,1/q\}$ for any $q\in \Q_{+}$ and such that $\pi^2=1,$ and also define 
$(\psi_\pi x)_q=x_{\pi(q)}.$
Letting $S_n=\psi_{\pi}\circ T_n\circ \psi_{\pi},$ and $A=\{(x_q)\colon x_1=1\},$ a straightforward calculation implies that $(T_nx)_1=x_n$ and $(S_nx)_1=x_{\pi(n)}$ thus, we deduce that \begin{equation*}
    \frac{1}{N}\sum_{n=1}^N\mu(T_n^{-1}A\cap S_n^{-1}A)=\frac{1}{N}\sum_{n=1}^{N}\left(\frac{1}{2}\1_{\pi(n)=n}+\frac{1}{4}\1_{\pi(n)\neq n}\right).
\end{equation*} We can now choose $\pi$ in a similar fashion as in \cite{Berend85} so that the last quantity does not converge. This also implies that 
$\frac{1}{N}\sum_{n=1}^N T_n\1_{A}\cdot S_n\1_A$ cannot converge in $L^2(\mu)$ (not even weakly).}

\subsection{Overview of the proof}

In order to prove Theorem~\ref{main theorem}, our first step is to 
use Lemma~\ref{fg actions characterization lemma} (see Section \ref{section_background} below) to translate our convergence problems into ones involving additive systems with completely additive functions in the iterates. This maneuver enables access to the machinery of box seminorms, which can bound our original averages.
After some further technical reductions, we show that it suffices to establish convergence for averages of the form \begin{equation}\label{eq_general form of average with additive transformations}
    \frac{1}{N}\sum_{n=1}^{N}\prod_{i=1}^\ell T_{i,1}^{a_1(n)}\ldots T_{i,m}^{a_m(n)}F_i 
\end{equation}
for some completely additive functions $a_1,\ldots, a_m$ and commuting measure-preserving transformations
$(T_{i,j})_{i\in[\ell], j\in[m]}$, both satisfying some additional pleasant properties (see Definitions~\ref{nice sequences},~\ref{defn_large_sets} and~\ref{collection of actions}).

To get the desired seminorm control, we employ a variant of the classical criterion of \Katai\ \cite{katai} (later rediscovered in \cite{bourgain_sarnak_ziegler}), which asserts that if a bounded arithmetic function $a$ taking values on a Hilbert space $(\Hilb, \norm{\cdot})$ satisfies $$\lim_{N\to\infty} \Bigg|\frac{1}{N}\sum_{n=1}^N \langle a(pn), a(qn)\rangle\Bigg|=0$$
for all distinct primes $p,q$,
then the averages $$ \Bigg\|\frac{1}{N}\sum_{n=1}^Na(n) \Bigg\|$$ converge to zero. This is a simple consequence of the \Turan-Kubilius inequality. 
For our purposes, we adapt a generalization 
of the \Turan-Kubilius inequality (first recorded in \cite{richter_PNT}) that allows us to replace the primes $p,q$ in the previous expression with numbers $\wt{p}, \wt{q}$
where $\wt{p}$ has exactly $\ell_1$ prime factors and $\wt{q}$ has exactly $\ell_2$ prime factors, for some $\ell_1,\ell_2\in \N$. The ability to choose $\ell_1, \ell_2$  to be distinct in this expression is crucial, as the bounds we get are trivial when $\ell_1=\ell_2$ (thus, for instance, the classical version of \Katai's criterion is ineffective in providing seminorm control).

In fact, we need a quantitative bound on the averages of the sequence $a(n)$ in terms of the correlations of its prime dilates. This bound is straightforward for logarithmic averages, since they are insensitive to small dilations in the scales. However, for \Cesaro\ averages it loses a factor of size $\min\{p,q\}$ without extra assumptions, which leads to terrible bounds.
To overcome this, we use the fact that the sequences we study satisfy some extra stability assumptions. More precisely, we show that averages of the form $$\frac{1}{N}\sum_{n=1}^N F(a_1(n),\ldots, a_k(n))$$oscillate very slowly, for a nice class of additive functions $a_1,\ldots, a_k$ and any bounded $F:\Z^k\to \Hilb$, where $\Hilb$ is a Hilbert space. This provides the extra input necessary to extend the argument above from the logarithmic averages to standard \Cesaro{} averages. This stability property is proven by applying known results of Tenenbaum \cite{MR3723446} from probabilistic number theory related to the distribution of $ \Omega$-like functions.

For the sake of exposition, let us suppose that we are dealing with a special case of \eqref{eq_general form of average with additive transformations} with $m=3,\ell=2$. That is, we want to prove convergence for
$$\frac{1}{N}\sum_{n=1}^N T_1^{a_1(n)}T_2^{a_2(n)}T_3^{a_3(n)}F_1\cdot R_1^{a_1(n)}R_2^{a_2(n)}R_3^{a_3(n)}F_2$$
for some commuting measure-preserving maps $T_1,T_2,T_3, R_1,R_2,R_3$, satisfying the nice assumptions of Definition~\ref{collection of actions}. Applying our orthogonality criterion, we can bound this average by the box seminorm\footnote{For the definition of the box seminorms, we refer our reader to Section \ref{section_background}.} $\nnorm{F_1}_{T_1, T_1R_1^{-1}}$. In later sections, we provide an example that illustrates how we arrive at these types of bounds.

The previous seminorm control holds for any system $(X,\X,\mu,T_1,T_2,T_3,R_1,R_2,R_3)$. However, the factor corresponding to this seminorm does not have any beneficial structure in general. Nonetheless, Host observed in \cite{Host-L^2-commuting-linearcase} that any system always has an extension (called a ``magic extension''), in which the corresponding factor admits a helpful description.  

Thus, we argue as follows: in order to prove convergence for our averages, it suffices to prove the convergence of the corresponding averages in this magic extension. Let $(\wt{X}, \wt{\mu},\wt{T}_1, \wt{T}_2,\wt{T}_3, \wt{R}_1, \wt{R}_2, \wt{R}_3)$ be this extension, then our averages become
$$\frac{1}{N}\sum_{n=1}^N \wt{T}_1^{a_1(n)}\wt{T}_2^{a_2(n)}\wt{T}_3^{a_3(n)}G_1\cdot \wt{R}_1^{a_1(n)}\wt{R}_2^{a_2(n)}\wt{R}_3^{a_3(n)}G_2.$$
If $G_1$ is orthogonal to the factor corresponding to the seminorm $\nnorm{\cdot}_{\wt{T}_1, \wt{T}_1\wt{R}_1^{-1}}$, then the average converges to zero by our seminorm bound, and we are done. Otherwise, we may assume that our function $G_1$ is measurable with respect to this factor. In the magic extension, we can approximate any function measurable with respect to this factor by linear combinations of functions of the form $H_1\cdot H_2$, where $H_1$ is invariant with respect to $\wt{T}_1$ and $H_2$ is invariant with respect to $\wt{T}_1\wt{R}_1^{-1}$ (i.e., the maps appearing in the seminorm). It suffices, therefore, to establish convergence in the case where $G_1$ is replaced with the product $H_1\cdot H_2$. Using the relations $\wt{T}_1H_1=H_1$ and $\wt{T}_1H_2=\wt{R}_1H_2$, we have to prove convergence for the averages $$\frac{1}{N}\sum_{n=1}^N \wt{T}_2^{a_2(n)}\wt{T}_3^{a_3(n)}H_1\cdot \wt{R}_1^{a_1(n)}\wt{T}_2^{a_2(n)}\wt{T}_3^{a_3(n)}H_2\cdot  \wt{R}_1^{a_1(n)}\wt{R}_2^{a_2(n)}\wt{R}_3^{a_3(n)}G_2,$$
where the nice properties of the actions are preserved.
We claim by induction that an average of this form converges for all systems 
and any choice of functions $H_1,H_2,G_2$, and then we are done.

To summarize, we started with a given average and passed to an extension to reduce our problem to an average that we claim is easier to understand.
Since we need to induct on the form that the average can take, we need a measure for the complexity of an average like \eqref{eq_general form of average with additive transformations}. It is not clear that our problem was simplified in the previous exposition, since the length of the average increased (we started with two measurable functions and we reduced to an average involving three functions). Nevertheless, we claim that the complexity of the average decreased. This is analogous to the PET induction scheme in the additive setting (via the van der Corput inequality), originally introduced by Bergelson in \cite{BergelsonPET}. 
To quantify this notion precisely, we introduce to any collection of actions (not necessarily satisfying the nice properties of Definition~\ref{collection of actions}) a triplet $(c_1,c_2,c_3)$ called the complexity vector, where:

i) $c_1$ counts the number of additive functions appearing in the iterates (i.e., $c_1=m$ in \eqref{eq_general form of average with additive transformations} and $m=3$ in both of the previous examples), ignoring the positions $j\in[m]$ that are constant in the sense that $T_{i,j}=T_{i',j}$ for all $i,i'\in[\ell]$.

ii) We identify the index $j$, among the nonconstant positions, that has the least number of distinct nonidentity transformations $T_{i,j}$ (thus, the most repetitions), and we let $c_2$ be the number of distinct nonidentity transformations in this position. In our example, we have two transformations in position $i$ (for all $i\in \{1,2,3\}$), so that $c_2=2$. In the average after the reduction, we note that $c_2$ has reduced to 1. Indeed, there is only one transformation corresponding to the position $a_1$ (namely, $\wt{R}_1$).

iii) If $j$ is the position we fixed in the definition of $c_2$, we look at the nonidentity transformation $T_{i,j}$ that repeats the least, and we let $c_3$ be the number of times this transformation appears. For instance, we have $c_3=1$ in the first example, since in position $a_1$, we have that both $T_1, R_1$ appear once (the same holds for the other positions $a_2,a_3$). Similarly, we have that the new average after the reduction has $c_3=2$, since in position $a_1$ (where $c_2$ is attained), we have only one transformation $\wt{R}_1$, appearing twice.

We order the triplets $(c_1,c_2,c_3)$ lexicographically. Naturally, we say that an average with complexity vector $(c_1,c_2,c_3)$ has lower complexity 
than an average with complexity vector $(c_1',c_2',c_3')$, when $(c_1,c_2,c_3)\prec (c_1',c_2',c_3')$. In the previous example, we started with an average with vector $(3,2,1)$ and reduced our problem to an average with complexity vector $(3,1,2)$. Since $(3,1,2)\prec (3,2,1)$, the complexity decreased.
 
We prove that our procedure $$\text{seminorm control } \to \text{ pass to magic extension }\to  \text{construct new average}$$ can be iterated
until we reach an average of complexity corresponding to Theorem~\ref{thm_PMET}.

\subsection{Open problems}

We conclude the introduction with some open problems.

A natural generalization of Theorem~\ref{main theorem} involves multiplicative actions that do not necessarily commute. We have seen that the theorem is false without any commutativity assumption, but a slight weakening of this notion may still be sufficient.

\begin{Conjecture}
    Let $(X,\X,\mu)$ be a probability space, and $S_1,\ldots, S_{\ell}$ be finitely generated multiplicative actions generating a nilpotent group. 
    Then, for any functions $F_1,\ldots,F_\ell\in L^\infty(\mu)$, the averages
    \begin{equation*}
       \frac{1}{N}\sum_{n=1}^N S_{1,n}F_1\cdot\ldots\cdot S_{\ell,n}F_\ell
    \end{equation*}
    converge in $L^2(\mu)$.
    \end{Conjecture}

Another generalization of the main theorem would be a strengthening in the mode of convergence, such as upgrading mean convergence to pointwise convergence. Unfortunately, the situation is substantially more complicated in the multiplicative setting. In fact, for the action $S_n=T^{\Omega(n)}$ on a probability space $(X,\X,\mu)$, it is known that pointwise convergence fails in non-atomic ergodic systems \cite{Loyd}. Here, $\Omega(n)$ is the completely additive function defined by $\Omega(p)=1$ for all $p\in \P$.
On the other hand, pointwise convergence does hold under the stronger assumption of unique ergodicity for continuous fucntions,\footnote{A {\em topological dynamical system} is a pair $(X,T)$, where $X$ is a compact metric space and $T\colon X\to X$ is continuous. The system is called {\em uniquely ergodic} if it admits a unique $T$-invariant Borel probability measure. Then, the corresponding measure-preserving system $(X,\X,\mu,T)$ is ergodic, making $\mu$ the only ergodic measure of the system.} and in fact, at every point of the space. More precisely, Bergelson and Richter \cite{Bergelson-Richter} proved that, in a uniquely ergodic system, the averages
\begin{equation*}
    \frac{1}{N}\sum_{n=1}^N F(T^{\Omega(n)}x)
\end{equation*}
converge for any continuous function $F$ and every point $x\in X$. 

\begin{Problem}
   Let $(X,T)$ be an invertible topological dynamical system. Under what assumptions do the averages \begin{equation*}
         \frac{1}{N}\sum_{n=1}^N F(T^{\Omega(n)}x) \cdot G(T^{2\Omega(n)}x)
    \end{equation*}
    converge for every $F,G\in C(X)$ and every $x\in X$?
\end{Problem}

Another interesting question arises from the interaction between additive and multiplicative systems.
\begin{Conjecture}
    Let $(X,\X,\mu)$ be a probability space, $T_1,\ldots, T_k$ be invertible commuting measure-preserving maps, and $S_1,\ldots, S_{\ell}$ be invertible commuting finitely generated multiplicative actions. Suppose $p_1,\ldots, p_k\in \Z[x]$. Then, the averages \begin{equation*}
         \frac{1}{N}\sum_{n=1}^N T_1^{p_1(n)}F_1\cdot\ldots \cdot T_k^{p_k(n)}F_k\cdot S_{1,n}G_1\cdot\ldots\cdot S_{\ell,n}G_{\ell}
    \end{equation*}converge in $L^2(\mu)$ for any functions $F_1,\ldots, F_k,G_1,\ldots, G_{\ell}\in L^{\infty}(\mu)$.
\end{Conjecture}

This problem has been studied only in the case $\ell=1$ and $S_{1,n}=T^{\Omega(n)}$ for some measure-preserving map $T$.
A result of this form with $k=1$ and $p_1(n)=n$ appears in \cite[Corollary 1.33]{Charamaras-multiplicative}, with some extra ergodicity assumptions that ensure convergence of the averages to the product of the integrals. A more general result, allowing arbitrary $k\in\N$ and general polynomials (subject to some mild restrictions), appears in \cite{XIAO_2026}.

This problem is interesting even in the case $T_1=\ldots= T_k$. A potential approach could involve combining the complexity reduction scheme in this article with the classical PET induction scheme used for polynomial ergodic averages, though it is not clear whether this approach is sufficient. For this method to work, some commutativity assumption between the additive and multiplicative actions might be necessary.

The final question involves the problem of joint ergodicity which has become intensely studied in recent years (see \cite{kuca_survey} for the history of the problem). In simplistic terms, joint ergodicity 
is the phenomenon of convergence of ergodic averages to the product of the integrals of the functions involved.
Since we are using additive averages to study multiplicative actions, the correct notion to use is that of \emph{pretentious ergodicity}, introduced in \cite{Charamaras-multiplicative}. We proceed to define this notion explicitly.

Let $(X,\X,\mu,S_1,\ldots,S_\ell)$ be a multiplicative system. We say that $S_1,\ldots,S_\ell$ are {\em jointly pretentiously ergodic (for $\mu$)} if for any $F_1,\ldots,F_\ell\in L^\infty(\mu)$, we have
    \begin{equation}\label{defn_JPE_eqn}
        \lim_{N\to\infty}\frac{1}{N}\sum_{n=1}^N S_{1,n}F_1\cdot\ldots\cdot S_{\ell,n}F_\ell = \int_X F_1\, d\mu \cdot\ldots\cdot \int_X F_\ell\, d\mu
    \end{equation}
in $L^2(\mu)$. In the case $\ell=1$, we call $S_{1,n}$ \emph{pretentiously ergodic}.\footnote{The notion of pretentious ergodicity was defined differently in \cite{Charamaras-multiplicative}, though our definition is equivalent due to \cite[Theorem A]{Charamaras-multiplicative}.}

In this setting, the next conjecture is the analog of the result of Berend-Bergelson \cite{Bergelson-Berend} for $\Z^k$-actions.

\begin{Conjecture}
       Let $(X,\X,\mu,S_1,\ldots,S_\ell)$ be a finitely generated multiplicative system. Then $S_1,\ldots,S_\ell$ are jointly pretentiously ergodic if and only if both of the following conditions are satisfied:
    \begin{enumerate}[label=(\arabic*)]
        \item\label{cond1} $S_iS_j^{-1}$ is pretentiously ergodic for all distinct $1\leq i,j\leq \ell$,
        \item\label{cond2} $S_1\times\cdots\times S_\ell$ is pretentiously ergodic for $\underbrace{\mu\times\ldots\times\mu}_{\ell\ \text{times}}$.
    \end{enumerate}
\end{Conjecture}

Similar conjectures can be formulated for more complicated averages, where each action is evaluated along sequences other than $a(n)=n$, but we do not state the problem in this generality.

\subsection{Notation} We use the standard notation $\N$ for the natural numbers, $\N_0:=\N\cup\{0\}$, $\Z$ for the integers, $\Q$ for the rational numbers, $\Q_+$ for the positive rationals, $\P$ for the primes, $\R$ for the real numbers, $\C$ for the complex numbers, and $\mathbb{S}^1$ for the complex unit circle. For any $C\geq 1$, we let $[C]:=\{1,\ldots,\floor{C}\}$, where $\floor{C}$ is the integer part of $C$.
Given a nonempty finite sets $A\subseteq\Z$, $B\subseteq\N$, we write
$$\E_{n\in A}f(n) := \frac{1}{|A|}\sum_{n\in A}f(n),
\qquad
\logE_{n\in B}f(n):= \frac{1}{\sum_{n\in B}\frac{1}{n}}
\sum_{n\in B}\frac{f(n)}{n},$$
for the \Cesaro{} average and the logarithmic average of a function $f$ over $A$, respectively.
For $u\in\N_0$ and $z\in\C$, we define
$$\CC^u z :=
\begin{cases}
z, & \text{if } u \text{ is even},\\
\overline{z}, & \text{if } u \text{ is odd}.
\end{cases}$$
In addition, we use the standard asymptotic notation. Thus, given two functions $f,g$, $f(x)=\Oh(g(x))$ and $f(x)\ll g(x)$ both mean that $|f(x)|\leq C|g(x)|$ for some absolute constant $C>0$ and $x$ large enough, while $\oh_{x\to\infty}(1)$ denotes a quantity that goes to $0$ as $x$ grows to infinity. Subscripts on the
symbols $\Oh$ and $\ll$ indicate the permitted dependence of the implicit constant.
Finally, given a probability space $(X,\X,\mu)$ and a sub-$\sigma$-algebra
$\CY\subseteq\X$, we denote by $\mathcal{E}_\mu(F|\CY)$
the conditional expectation of a function $F\in L^1(\mu)$ with respect to $\mathcal{Y}$. 

\medskip

\noindent {\bf{Acknowledgment.}} We thank Nikos Frantzikinakis for suggesting the problem to us.

\medskip

\noindent{{\bf{AI statement}}}: An analysis by OpenAI's GPT-6 Astra model identified errors in an earlier version of this article--specifically in previous versions of Proposition~\ref{prop_Katai variant}, the reduction procedure in Section~\ref{complexity reduction procedure subsection}, and Lemma~\ref{last technical lemma}--and provided the references \cite{Ford} and \cite{MR3723446}, which helped resolve the first issue. 
The errors were corrected by the authors, who remain responsible for the content of the manuscript.

\section{Background}\label{section_background}

\subsection{Measure-preserving systems, factors and extensions}

Let $G$ be a countable semigroup. A \emph{(measure-preserving) $G$-system} is a standard probability space $(X,\X,\mu)$ equipped with a collection of measure-preserving transformations $T_g$, so that $G$ acts on $X$ through the transformations $T_g$: for any $g_1,g_2\in G$ we have $T_{g_1}\circ T_{g_2}=T_{g_1g_2}$.
In this paper, we only concern ourselves with actions of the group $(\Z^k,+
)$ and the semigroup $(\N,\times)$ (or $(\Q_{+},\times)$ for invertible actions).

We say the system $(Y,\mathcal{Y},\nu,(S_{g})_{g\in G})$ is a {\em factor} of $(X,\mathcal{X},\mu,(T_g)_{g\in G})$, and the latter is an {\em extension} of the former, if there exist $X'\subseteq X$, $Y'\subseteq Y$ of full measure that are invariant under the actions $(T_g)_{g\in G}$ and $(S_g)_{g\in G}$ respectively and a map $\pi:X'\to Y'$ such that $\nu=\mu\circ \pi^{-1}$ and $\pi\circ T_g(x)=S_g\circ \pi(x)$ for all $x\in X'$ and $g\in G$. A factor of the system $(X,\mathcal{X},\mu,(T_g)_{g\in G})$ corresponds to a $(T_g)_{g\in G}$-invariant sub-$\sigma$-algebra of $\mathcal{X}$ (in the above example this $\sigma$-algebra is $\pi^{-1}(\mathcal{Y})$). We record here that, with this identification, the space $L^{\infty}(Y)$ can be viewed as a subalgebra of $L^{\infty}(X)$ that is invariant under the $ G$-action on the probability space $(X,\X,\mu)$.

\subsection{Characterization of finitely generated multiplicative actions}

We provide a result of Frantzikinakis \cite{frantzikinakis2025decomposition} on the characterization of a finitely generated multiplicative action, which is an essential ingredient for the proof of Theorem~\ref{main theorem}.
\begin{Definition}\label{nice sequences}
Given a completely additive function $a\colon\N\to\N_0$, we define its {\em prime support} by
$$\CP(a) = \{p\in\P\colon a(p)\neq 0\}.$$
We say that the completely additive functions $a_1,\ldots,a_m\colon\N\to\N_0$ are {\em nice} if for all $j\in[m]$, $a_j(\P)\subseteq\{0,1\}$ and $\CP(a_j)\neq\emptyset$, and for all distinct $j,j'\in[m]$, $\CP(a_j)\cap\CP(a_{j'})=\emptyset$.
\end{Definition}

\begin{Lemma}[{Cf. \cite[Lemma 2.7]{frantzikinakis2025decomposition}}]
\label{fg actions characterization lemma}
    Let $\xms$ be a multiplicative system. Then $\xms$ is finitely generated if and only if there exist $m\in\N$, commuting invertible measure-preserving transformations $T_1,\ldots,T_m\colon X\to X$ with $T_i\neq T_j$ for $i\neq j$, and finitely generated completely additive functions $a_1,\ldots,a_m\colon\N\to\N_0$ that are nice, such that for all $n\in\N$,
    $$S_n = T_1^{a_1(n)}T_2^{a_2(n)}\ldots T_m^{a_m(n)}.$$
\end{Lemma}

In fact, in the proof of Theorem~\ref{main theorem}, as we will see, we can restrict to the case where all the actions have the above form for the same nice functions $a_1,\ldots,a_m$, and in addition, that each of these functions has large prime support in the following sense.

\begin{Definition}\label{defn_large_sets}
    We say that a set $\CP\subseteq\P$ is {\em large} if $H(\CP)=\infty$, and {\em thin} otherwise.
\end{Definition}

\subsection{Host-Kra seminorms, box seminorms and factors}

Here, we recall the definition of box seminorms, which were first introduced by Host \cite{Host-L^2-commuting-linearcase} to give another proof of the convergence of linear ergodic averages for commuting transformations.
Let $(X,\mu, T_1,\ldots,  T_k)$ be a $\Z^k$-system and
$R_1,\ldots, R_s\in \langle T_1,\ldots, T_k \rangle$ be a collection of measure-preserving transformations generated by the maps $T_i$.
We define the box seminorms on $L^{\infty}(\mu)$ along the directions $R_1,\ldots, R_s$ inductively by 
\begin{equation*}
     \nnorm{F}_{R_1}^2=\lim_{M\to\infty}\E_{m_0,m_1\in [M]} \int_X R_1^{m_0}F\cdot R_1^{m_1}\overline{F} \, d\mu
 \end{equation*} 
 and 
 \begin{equation*}
    \nnorm{F}_{R_1,\ldots,R_s}^{2^s}=\lim_{M\to\infty}\E_{m_0,m_1\in [M]}\nnorm{R_s^{m_0}F\cdot R_s^{m_1}\overline{F} }_{R_1,\ldots,R_{s-1}}^{2^{s-1}}.
 \end{equation*}

The existence of the limits above can be found, for example, in a more general setting in \cite{tsinas&friends}, and it can be shown that the iterated limit in this definition can be replaced by a simultaneous limit or by any limit along any \Folner\ sequence. We remark here that the first order seminorm has a closed form \begin{equation*}
    \nnorm{F}_{R_1}=\norm{\mathcal{E}_{\mu}(F|\mathcal{I}(R_1))}_2,
\end{equation*} which follows from the mean ergodic theorem.

It can also be proven that $\nnorm{F}_{R_1,\ldots, R_s}$ are indeed seminorms. Expanding the definition, we can write \begin{multline}\label{eq_expanding the seminorm}
  \nnorm{F}_{R_1,\ldots,R_s}^{2^s}=\lim_{M_s\to\infty}\dots \lim_{M_1\to\infty}\E_{m_{1,0},m_{1,1}\in [M_1]}\dots   \E_{m_{s,0},m_{s,1}\in [M_s]} \\
  \int_X \prod_{\underline{\e}\in \{0,1\}^s}\mathcal{C}^{|\underline{\e}|}R_1^{m_1{\e}_1}\ldots R_s^{m_s{\e}_s}F \, d\mu,
\end{multline}
where $$m_i{\e}_i:=\begin{cases}m_{i,0}, \ \text{ if } {\e}_i=0,\\
m_{i,1}, \ \text{ if } {\e}_i=1.
\end{cases}$$

As before, the iterated limits in \eqref{eq_expanding the seminorm} can be replaced with a simultaneous limit or any other limit along a \Folner\ sequence in $\Z^{2s}$ without affecting the value of the limits.

When $R_1=\ldots =R_s$, the box seminorm is simply the classical Host-Kra seminorm $\nnorm{\cdot}_{s,R_1}$ corresponding to $R_1$. These are defined inductively by\footnote{In general, defining the seminorms at each inductive step through an average over one variable or an average over two variables leads to equivalent seminorms, but we prefer to use the version with two shift parameters $m_0, m_1$ since this is more in line with the expressions appearing in our proofs.} \begin{align*}
    \nnorm{F}_{1,R_1}^2&=\lim_{M\to\infty}\E_{m\in [M]} \int_X F\cdot R_1^m \overline{F}\, d\mu \\
    \nnorm{F}_{s,R_1}^{2^s}&=\lim_{M\to\infty}\E_{m\in [M]} \nnorm{F\cdot R_1^{m}\overline{F}}_{s-1,R_1}^{2^{s-1}}
\end{align*}
and the averages can be taken again along any \Folner\ sequence.

It was proven in \cite{Host-L^2-commuting-linearcase} that the box seminorm $\nnorm{\cdot}_{R_1,\ldots, R_s}$ gives rise to an invariant $\sigma$-algebra (i.e. a factor) $\CZ_{R_1,\ldots,R_s}(X)$ characterized by the following property\begin{equation}\label{eq_characteristic property of factors}
    \nnorm{F}_{R_1,\ldots,R_s}=0\iff \mathcal{E}_{\mu}(F\big| \CZ_{R_1,\ldots,R_s}(X) )=0. 
\end{equation} 
To define this factor, see the construction in \cite[Theorem 2.4]{Tao-Ziegler-concatenation}. 

\subsection{Magic extensions}

In order to extract some valuable information from these seminorm bounds, we will use the machinery of magic extensions. 
\begin{Definition}\label{defn magic}
    Let $(X,\mu, T_1,\ldots, T_k)$ be a $\Z^k$-system. We call this system \emph{magic with respect to the transformations $R_1,\ldots, R_s\in \langle T_1,\ldots, T_k\rangle$} if the factor $\mathcal{Z}_{R_1,\ldots, R_s}$ satisfies $$\mathcal{Z}_{R_1,\ldots, R_s}=\mathcal{I}(R_1)\vee \dots\vee \mathcal{I}(R_s),$$
    where $\mathcal{I}(R_1)\vee \dots\vee \mathcal{I}(R_s)$ is the smallest $\sigma$-algebra containing $\I(R_1),\ldots,\I(R_s)$.
\end{Definition}
It can be shown that $\mathcal{Z}_{R_1,\ldots, R_s}$ is a factor of the original system. A function with $F$ that is measurable with respect to $\mathcal{Z}_{R_1,\ldots, R_s}$ in a magic system can be approximated in $L^2(\mu)$ by linear combinations of functions of the form $g_1\cdot\ldots\cdot g_s$, where $g_i$ is invariant under $R_i$.

The following is a generalization of \cite[Theorem 2]{Host-L^2-commuting-linearcase}.

\begin{Proposition}[{\cite[Proposition 2.8]{fra_ku_hardy_corners}}]
\label{prop_Nikos_and_Borys}
Let $(X, \mu, T_1, \dots , T_k)$ be a $\Z^k$-system,
and suppose that the transformations $R_1, \dots, R_s \in \langle T_1, \dots , T_k\rangle $ generate $T_1, \dots, T_s$ for some $s\in [k]$. Then
$(X,  \mu, T_1, $ $ \dots, T_k)$ admits an extension $(X^{*},\mu^{*}, T^{*}_1, \dots , T^{*}_k
 )$ that is magic with respect to
$R_1^\ast, \dots , R_s^\ast$.
\end{Proposition}

\subsection{Background on probabilistic number theory}
We start by gathering some standard facts from probability theory.
We recall that a discrete random variable $Z$ follows the Poisson distribution with parameter $\lambda>0$, and we write 
$Z\sim\operatorname{Pois}(\lambda),$ if the probability mass function is given by
$$\mathbf{P}(Z=k)=\frac{e^{-\lambda}\lambda^k }{k!},\quad k\in\N_0.$$
For two discrete probability measures $\mu, \nu$ defined on a countable set $S,$ we define their \emph{total variation distance} by
$$d_{\mathrm{TV}}(\mu,\nu) = \sup_{E\subseteq S} |\mu(E)-\nu(E)| = \frac{1}{2}\sum_{s\in S}|\mu(\{s\})-\nu(\{s\})|.$$
In case $\mu_1, \mu_2$ denote the distributions of the random variables $Z_1,Z_2$ respectively, i.e., $\mu_i(E)=\mathbf{P}(Z_i\in E),\; i=1,2,$ we denote the former distance by $d_{\mathrm{TV}}(Z_1,Z_2)$ as well.
Finally, for two discrete probability distributions $\mu$ and $\nu$ defined on a countable set $S$, we will make use of the Kullback–Leibler (KL) divergence of $\nu$ from $\mu,$ defined as
$$D_{\mathrm{KL}}(\mu
\,\|\,\nu)=\sum_{s\in S}\mu(\{s\})\log\left(\frac{\mu(\{s\})}{\nu(\{s\})}\right).$$

In order to prove a crucial stability lemma for the \Cesaro{} averages of sequences evaluated along certain nice completely additive functions (see Lemma~\ref{lemma_functions_satisfying_stability}), we rely on a result of Tenenbaum \cite{MR3723446} (see Theorem~\ref{T: theorem tanenbaum} below) asserting that for such functions $a_1,\ldots,a_m\colon\N\to\N_0$ and for $N\in\N$, if $n$ is chosen uniformly at random from $[N]$, then $a_1(n),\ldots,a_m(n)$ behave as independent Poisson random variables with parameters $H(\CP(a_1)\cap[2,N]),\ldots, H(\CP(a_m)\cap[2,N]),$ where 
for a set $A\subseteq\N$, we denote
$$
    H(A) := \sum_{n\in A}\frac{1}{n}.
$$

Given $\CP\subseteq\P$, we define the completely additive function $\Omega_\CP$ given by $\Omega_\CP(p)=\1_{\CP}(p)$. We remark that if $a\colon\N\to\N_0$ is a nice completely additive function, then $a = \Omega_{\CP(a)}$.

The following theorem follows from the formula in \cite[Equation (2.23)]{MR3723446}. Strictly speaking, the formula is stated for the additive counterparts $\omega_{\CP}$, but the author remarks that the joint distribution for the functions $\Omega_\CP$ is immediate. See also the discussion after \cite[Theorem 1]{Ford}, as well as equations $(1.3)$ and $(1.4)$ in the last article.

\begin{Theorem}[\cite{MR3723446}]\label{T: theorem tanenbaum}
Let $\CP_1,\ldots,\CP_m$ be nonempty, pairwise disjoint sets of primes contained in $[2,N]$, $c>0$, and $\mathbf{P}_N$ be the uniform probability measure on $[N]$.
Uniformly for integers $k_1,\ldots,k_m$ satisfying
$$
    c \leq \frac{k_j}{H(\CP_j)} \leq 2-c,
    \qquad j\in[m],
$$
we have
$$
\begin{aligned}
&\mathbf{P}_N\bigl(
    \Omega_{\CP_j}(n)=k_j,\ j\in[m]
\bigr)
\\
&\quad =
\left(
    1+\Oh\left(
        \sum_{j=1}^m \frac1{\sqrt{H(\CP_j)}}
    \right)
\right)
\left(
    \prod_{j=1}^m
    \frac{H(\CP_j)^{k_j}}{k_j!}e^{-k_j}
\right)
\frac1N
\sum_{n\leq N}
\prod_{j=1}^m
\left(
    \frac{k_j}{H(\CP_j)}
\right)^{\Omega_{\CP_j}(n)} .
\end{aligned}
$$
In particular,
$$
\begin{aligned}
&\mathbf{P}_N\bigl(
    \Omega_{\CP_j}(n)=k_j,\ j\in[m]
\bigr)
\\
&\quad =
\prod_{j=1}^m
\frac{H(\CP_j)^{k_j}}{k_j!}e^{-H(\CP_j)}
\exp\left(
    \Oh\left(
        \sum_{j=1}^m
        \left(
            \frac{|k_j-H(\CP_j)|}{H(\CP_j)}
            +\frac1{\sqrt{H(\CP_j)}}
        \right)
    \right)
\right).
\end{aligned}
$$
\end{Theorem}

This implies the following corollary (see also {\cite[Equation~(1.4)]{Ford}} for a similar for the functions $\omega_{P_j}$).
\begin{Corollary}\label{C:tenenbaum}
Let $\CP_1,\ldots,\CP_m$ be pairwise disjoint nonempty sets of primes
contained in $[2,N]$.
For each $j\in[m]$, let
$$
Z_j\sim\operatorname{Pois}\bigl(H(\CP_j)\bigr)
$$
be independent Poisson random variables. Then
$$
d_{\mathrm{TV}}\left(
    \bigl(\Omega_{\CP_1}(n),\ldots,\Omega_{\CP_m}(n)\bigr),
    (Z_1,\ldots,Z_m)\right)
\ll_m
\sum_{j=1}^m \frac{1}{\sqrt{H(\CP_j)}},
$$
where $n$ is chosen uniformly at random from $[N]$.
\end{Corollary}

For the proof of \cref{lemma_functions_satisfying_stability} we also need a simple upper bound on the total variation distance of two Poisson random variables. By Pinsker's inequality, for $\lambda,\mu>0$,
$$
d_{\mathrm{TV}}(\operatorname{Pois}(\lambda),\operatorname{Pois}(\mu))^2
\leq \frac12
D_{\mathrm{KL}}\bigl(\operatorname{Pois}(\lambda)
\,\|\,\operatorname{Pois}(\mu)\bigr).
$$
For Poisson distributions,
$$
D_{\mathrm{KL}}\bigl(\operatorname{Pois}(\lambda)
\,\|\,\operatorname{Pois}(\mu)\bigr)
=
\lambda\log\frac{\lambda}{\mu}-\lambda+\mu.
$$
Consequently,
\begin{equation}\label{eq: Pinsker argument}
    d_{\mathrm{TV}}\bigl(
    \mathrm{Pois}(\lambda),
    \mathrm{Pois}(\mu)\bigr)
    \ll
    \frac{|\lambda-\mu|}{\sqrt{\min\{\lambda,\mu\}}},
\end{equation}
provided $|\lambda-\mu|=\Oh(1)$ and
$\min\{\lambda,\mu\}\to\infty$.

\section{Preparatory lemmas}

\subsection{A variant of \Katai's orthogonality criterion}

In order to establish that the average in \eqref{main average general} is controlled by a box seminorm, we rely on a variant of \Katai's orthogonality criterion. This variant, which we also view as a multiplicative analogue of the van der Corput inequality, is inspired by the work of Bergelson and Richter \cite{Bergelson-Richter}.

Given a set $\CP\subseteq\P$, and $M, k\in\N$, we define the set 
\begin{equation}\label{almost primes finite set}
    \CP^k(M) := \{n=p_1\cdot\ldots\cdot p_k\colon \;p_1,\ldots,p_k\in\CP\cap[M]~\text{are pairwise distinct}\}   
\end{equation}
For simplicity, in the case $k=1$, we denote the set $\CP^{1}(M)$ as $\CP(M)$, and this is just the set of primes in $\CP$ up to $M$.

\begin{Proposition}\label{prop_Katai variant}
Let $(\Hilb,\|\cdot\|)$ be a Hilbert space, $\CP\subseteq\P$ be a large set, and $f\colon\N\to\Hilb$ be a bounded function. If $f$ satisfies
\begin{equation}\label{eq_stability_of_averages_assumption}
    \lim_{N\to\infty}\Big\|\E_{n\in[N]}f(rn) - \E_{n\in[N/r]}f(rn)\Big\| = 0
\end{equation}
for any fixed $r\in\N$,
then for any $K\in\N$, we have
\begin{multline}\label{eq_Katai variant}
    \limsup_{N\to\infty}\Big\|\E_{n\in[N]} f(n)\Big\|^2 \\
    \leq \limsup_{M\to\infty}\limsup_{N\to\infty}\E_{k_1,k_2\in[K]}\logE_{\substack{r_1\in\CP^{k_1}(M) \\ r_2\in\CP^{k_2}(M)}} \E_{n\in[N]}\langle f(r_1n),f(r_2n)\rangle.
\end{multline}
\end{Proposition}

The reader familiar with the \Katai\ orthogonality criterion will recognize similarities between its statement and the last proposition. The difference is that instead of averaging over primes, we average over $k_1,k_2$ almost primes, that is, numbers that are the product of a controlled number of primes. Averaging over almost primes is an essential maneuver in this work. In contrast, an application of the \Katai\      
criterion does not improve over the trivial bound in our averages.

The next lemma provides a class of functions that satisfies the stability of averages condition \eqref{eq_stability_of_averages_assumption}.
\begin{Lemma}\label{lemma_functions_satisfying_stability}
Let $(\Hilb,\|\cdot\|)$ be a Hilbert space, $\Phi\colon\N_0^m\to\Hilb$ be a $1$-bounded function, and $a_1,\ldots,a_m\colon\N\to\N_0$ be a collection of nice completely additive functions with $\CP(a_j)$ large for all $j\in[m]$. Then the function $f(n) = \Phi(a_1(n),\ldots,a_m(n))$ satisfies \eqref{eq_stability_of_averages_assumption}.
\end{Lemma}

Combining Proposition~\ref{prop_Katai variant} with Lemma~\ref{lemma_functions_satisfying_stability} we immediately obtain the following.

\begin{Corollary}\label{cor_Katai variant}
    Let $(\Hilb,\|\cdot\|), \Phi$, and $a_1,\ldots,a_m$ be as in Lemma~\ref{lemma_functions_satisfying_stability} and denote $\mathbf{a}=(a_1,\ldots,a_m)$. Let also $\CP\subseteq\P$ be a large set. Then for any $K\in\N$,
    \begin{multline}\label{eq_Katai variant cor}
        \limsup_{N\to\infty}\Big\|\E_{n\in[N]}\Phi(\mathbf{a}(n))\Big\|^2 \\
        \leq \limsup_{M\to\infty}\limsup_{N\to\infty}\E_{k_1,k_2\in[K]}\logE_{\substack{r_1\in\CP^{k_1}(M) \\ r_2\in\CP^{k_2}(M)}}\E_{n\in[N]}\langle \Phi(\mathbf{a}(r_1n)),\Phi(\mathbf{a}(r_2m))\rangle.
    \end{multline}
\end{Corollary}

It remains to prove Proposition~\ref{prop_Katai variant} and Lemma~\ref{lemma_functions_satisfying_stability}.
We start by proving Proposition~\ref{prop_Katai variant}, for which
we require some preliminary lemmas. The first one is a variant of the \Turan-Kubilius inequality, which was originally stated and proven by Bergelson and Richter \cite{Bergelson-Richter} for sequences in $\C$. It extends the classical \Turan-Kubilius inequality that involves averaging over primes to averages over integers that share few divisors on average. 
The proof of the lemma for an arbitrary Hilbert space follows similar reasoning, so the proof is omitted.

\begin{Lemma}[Cf.~{\cite[Proposition 2.1]{Bergelson-Richter}}]
\label{br_lemma1}
Let $(\Hilb,\|\cdot\|)$ be a Hilbert space, $B\subseteq\N$ be finite and non-empty and $f\colon\N\to\Hilb$ be 1-bounded. Then we have 
$$\limsup_{N\to\infty}\Big\|\E_{n\in[N]}f(n) - \logE_{m\in B}\E_{n\in[N/r]}f(rn)\Big\| \ll \Big(\logE_{r,n\in B}(r,n)-1\Big)^{1/2},$$
where $(r,n)$ denotes the greatest common divisor of $r$ and $n$.
\end{Lemma}

We will take $B$ to be a set of almost primes coming from a fixed $\CP\subseteq \P$, so we need an estimate for the right-hand side in this case.

\begin{Lemma}\label{lemma_bounding double log average of the gdc-1}
Let $\CP\subseteq\P$ be large. For every fixed $k\in\N$ and all sufficiently large $M$, we have
$$\logE_{r,n\in\CP^k(M)} (r,n)-1 \ll_k \frac{1}{H(\CP(M))}$$
In particular, for all $k\in\N$, we have
$$\lim_{M\to\infty}\logE_{r,n\in\CP^k(M)}(r,n)-1 = 0.$$
\end{Lemma}

\begin{proof}
We start by writing $H_j=H(\CP^j(M))$ for $j\in[k]$, and we have
$$
    \logE_{r,n\in\CP^k(M)}(r,n)-1
    =
    \frac{1}{H_k^2}\sum_{j=1}^k\sum_{\substack{r,n\in\CP^k(M) \\ (r,n)\in\CP^j(M)}}\frac{(r,n)-1}{rn}.
$$
For each $j\in[k]$, writing $d=(r,n)$ and $r=ds$, we obtain
\begin{align*}
    & \sum_{\substack{r,n\in\CP^k(M) \\ (r,n)\in\CP^j(M)}}\frac{(r,n)-1}{rn}
    =
    \sum_{n\in\CP^k(M)}\frac{1}{n} \sum_{\substack{d\mid n\\d\in\CP^j(M)}} \left(1-\frac{1}{d}\right) \sum_{\substack{s\in\CP^{k-j}(M)\\(s,n)=1}}\frac{1}{s}
    \leq 
    \binom{k}{j}H_kH_{k-j},
\end{align*}
where we used that every $n\in\CP^k(M)$
has exactly $\binom{k}{j}$ divisors belonging
to $\CP^j(M)$.
Thus, 
$$
    \logE_{r,n\in\CP^k(M)}(r,n)-1 
    \leq \frac{1}{H_k^2}\sum_{j=1}^k \binom{k}{j}H_k H_{k-j}.
$$
For every fixed $j\in\N$, we have
$$
    H_j = \frac{H_1^j}{j!}(1+\Oh_j(H_1^{-2})).
$$
Indeed, expanding
$$
    H_1^j = \sum_{p_1,\ldots,p_j\in\CP\cap[M]} \frac{1}{p_1\cdot\ldots\cdot p_j},
$$
we have that the contribution from ordered $j$-tuples of primes with at least two equal coordinates
is at most
$$
    \binom{j}{2}H_1^{j-2}\sum_{p\in\CP(M)}\frac{1}{p^2}
    \ll_j H_1^{j-2},
$$
while the contribution from ordered $j$-tuples with pairwise distinct coordinates is $j! H_j$ since there are exactly $j!$ permutations of the set $\{p_1,\ldots,p_j\}$.

Consequently, for $j\in[k]$,
$$
    \frac{H_{k-j}}{H_k}
    =
    \frac{k!}{(k-j)!H_1^j}
    (1+O_k(H_1^{-2})).
$$
It follows that
\begin{align*}
    \logE_{r,n\in\CP^k(M)}(r,n)-1
    \leq \sum_{j=1}^k \binom{k}{j}\frac{k!}{(k-j)!H_1^j}(1+O_k(H_1^{-2})) 
    = \frac{k^2}{H_1}+O_k(H_1^{-2})
    \ll_k \frac{1}{H_1},
\end{align*}
concluding the proof.
\end{proof}

\begin{proof}[Proof of Proposition~\ref{prop_Katai variant}]
  We start by applying Lemma~\ref{br_lemma1} with $B=\CP^k(M)$ for any $k=1,\ldots,K$, and Lemma~\ref{lemma_bounding double log average of the gdc-1} to get
    $$\lim_{M\to\infty}\limsup_{N\to\infty}\Big\|\E_{n\in[N]}f(n) - \logE_{r\in\CP^k(M)}\E_{n\in[N/r]}f(rn)\Big\| = 0.$$
    Thus, for any $K\in\N$,
    \begin{multline*}
        \limsup_{M\to\infty}\limsup_{N\to\infty}\Big\| \E_{n\in[N]}f(n) - \E_{k\in[K]}\logE_{r\in\CP^k(M)}\E_{n\in[N/r]}f(rn)\Big\| \\
        \leq \limsup_{M\to\infty}\limsup_{N\to\infty}\E_{k\in[K]}\Big\|\E_{n\in[N]}f(n) - \logE_{r\in\CP^k(M)}\E_{n\in[N/r]}f(rn)\Big\| 
        = 0.
    \end{multline*}

This shows that it suffices to prove \eqref{eq_Katai variant} when substituting the averages in the left-hand side with $\E_{k\in[K]}\logE_{r\in\CP^k(M)}\E_{n\in[N/r]}f(rn)$. Now, we apply
\eqref{eq_stability_of_averages_assumption} to deduce that 
\begin{equation*}
    \lim_{N\to\infty}\Big\|\E_{n\in[N]}f(rn)-\E_{n\in[N/r]}f(rn)\Big\|=0.
\end{equation*}
for any fixed $r$. It follows that for fixed $M,K$, we have
\begin{align*}
    \Big\|\E_{k\in[K]}\logE_{r\in\CP^k(M)}\E_{n\in[N/r]}f(rn)\Big\|^2
    & = \Big\|\E_{k\in[K]}\logE_{r\in\CP^k(M)}\E_{n\in[N]}f(rn)\Big\|^2 +\oh_{N\to\infty}(1)\\
    & \leq \E_{n\in[N]}\Big\|\E_{k\in[K]}\logE_{r\in\CP^k(M)}f(rn)\Big\|^2+\oh_{N\to\infty}(1),
\end{align*}by the Cauchy-Schwarz inequality.
By expanding the square and rearranging, we see that the last expression is equal to
$$\E_{k_1,k_2\in[K]}\logE_{r_1\in\CP^{k_1}(M), r_2\in\CP^{k_2}(M)}\E_{n\in[N]} \langle f(r_1n),f(r_2n)\rangle+\oh_{N\to\infty}(1).$$
This concludes the proof of the proposition. 
\end{proof}

Now we shall prove Lemma~\ref{lemma_functions_satisfying_stability}.
\begin{proof}[Proof of Lemma~\ref{lemma_functions_satisfying_stability}]
For $q\in\N$ and for each $j\in[m]$, we define $a_{q,j}(n)=a_j(qn)$. For convenience we denote $\mathbf{a}_q(n) = (a_{q,1}(n),\ldots,a_{q,m}(n))$.
We consider the probability distribution of the random vector $\mathbf{a}_q(n)$ when $n$ is uniformly chosen at random from $[N]$. 
Namely, we define the probability measure $\nu_{q,N}$ through the formula
$$
  \nu_{q,N}(E)=\E_{n\in [N]}\1_{E}(\mathbf{a}_q(n))  \qquad \text{for every }E\subseteq \Z^m.
$$
Note that for the bounded function $\Phi$ in our statement and for any $N\in \N$, we have
$$\E_{n\in[N]}\Phi(\mathbf{a}_q(n)) = \int_{\Z^m} \Phi(\mathbf{n})\ d\nu_{q,N}(\mathbf{n}).$$
Thus, for $r\in \N$, using the triangle inequality, we have
\begin{align*}
    \Big\|\E_{n\in[N]} \Phi(\mathbf{a}_q(n)) - \E_{n\in[N']} \Phi(\mathbf{a}_q(n))\Big\| 
    & = \bigg\| \int_{\Z^m} \Phi \, d\nu_{q,N}  - \int_{\Z^m} \Phi \, d\nu_{q,N'} \bigg\| \\
    & \leq 2\|\Phi\|_\infty d_{\mathrm{TV}}(\nu_{q,N},\nu_{q,N'}),
\end{align*}
where $N'=\floor{N/r}$.
To conclude the proof, it suffices to show that for any $q\in\N$,
$$d_{\mathrm{TV}}(\nu_{q,N},\nu_{q,N'}) = \oh_{N\to\infty}(1).$$

Since $\mathbf{a}_q(n) = \mathbf{a}_1(n) + \mathbf{a}_1(q)$, $\nu_{q,N}$ and $\nu_{q,N'}$ are obtained from $\nu_N:=\nu_{1,N}$ and $\nu_{N'}:=\nu_{1,N'}$ respectively, by translation by the same vector $\mathbf{a}_1(q)$. Since the total variation distance is invariant under translations, it is enough to show that 
$$d_{\mathrm{TV}}(\nu_N,\nu_{N'}) = \oh_{N\to\infty}(1).$$

Writing $\CP_j=\CP(a_j)$, Corollary~\ref{C:tenenbaum} applies and we deduce that 
\begin{align*}
    d_{\mathrm{TV}}(\nu_N, \mathbf{Z}_N) & \ll \sum_{j=1}^m \frac{1}{\sqrt{H\left(\CP_j\cap[2,N] \right)}}, \\
    d_{\mathrm{TV}}(\nu_{N'}, \mathbf{Z}_{N'}) & \ll \sum_{j=1}^m \frac{1}{\sqrt{H\left(\CP_j\cap[2,N'] \right)}},
\end{align*}
where $\mathbf{Z}_N=(Z_{1,N},\ldots, Z_{m,N})$ and $\mathbf{Z}_{N'}=(Z_{1,N'},\ldots, Z_{m,N'})$ are two random vectors, each one consisting of independent Poisson random variables $Z_{j,N}$ and $Z_{j,N'}$  with parameters $H(\CP_j\cap [2,N])$ and $H(\CP_j\cap [2,N'])$ respectively. 
By assumption, $H(\CP_j\cap [2,N])\to \infty$ and $H(\CP_j\cap [2,N'])\to \infty$ as $N\to\infty$. Thus,
\begin{equation*}
      d_{\mathrm{TV}}(\nu_N, \mathbf{Z}_N)= \oh_{N\to\infty}(1), \quad
      d_{\mathrm{TV}}(\nu_{N'}, \mathbf{Z}_{N'})= \oh_{N\to\infty}(1).
\end{equation*}
The triangle inequality for the total variation distance implies that 
\begin{equation*}
    d_{\mathrm{TV}}(\nu_N,\nu_{N'}) \leq d_{\mathrm{TV}}(\mathbf{Z}_N,\mathbf{Z}_{N'})+\oh_{N\to\infty}(1).
\end{equation*}
Therefore, it suffices to show that $d_{\mathrm{TV}}(\mathbf{Z}_N,\mathbf{Z}_{N'})$ tends to zero as $N\to \infty$.

Note that 
$$d_{\mathrm{TV}}(\mathbf{Z}_N,\mathbf{Z}_{N'})\leq \sum_{j=1}^m d_{\mathrm{TV}}(Z_{j,N},Z_{j,N'})$$
using the independence of the coordinates within each of the vectors $\mathbf{Z}_{N}$ and $\mathbf{Z}_{N'}$.

Since we have Poisson random variables, it follows from \eqref{eq: Pinsker argument} 
\begin{align*}
    d_{\mathrm{TV}}(Z_{j,N},Z_{j,N'}) 
    & \leq \frac{|H(\CP_j\cap [2,N])-H(\CP_j\cap [2,N'])|}{\sqrt{H(\CP_j\cap [2,N'])}} \\
    & \leq \frac{1}{\sqrt{H(\CP_j\cap [2,N'])}}\sum_{N/r\leq p\leq N}\frac{1}{p} \\
    & \ll \frac{\log r}{\log N\sqrt{H(\CP_j\cap [2,N'])}}
    + \oh_{N\to\infty}(1)=\oh_{N\to\infty}(1),
\end{align*}
where we used Mertens' theorem to bound the prime sum.
It follows that for every $r\in\N$ we have $d_{\mathrm{TV}}(\mathbf{Z}_N,\mathbf{Z}_{N'})=\oh_{N\to\infty}(1)$, concluding the proof.
\end{proof}

We conclude this subsection with a result illustrating how Corollary~\ref{cor_Katai variant} will be used to bound the $L^2$-norm of the average in \eqref{main average general}.

\begin{Lemma}\label{Katai for actions}
    Let $(X,\mu, S_1,\ldots, S_\ell)$ be a finitely generated multiplicative system, $m\in\N$, and $a_1,\ldots,a_m\colon\N\to\N_0$ be nice finitely generated completely additive functions, such that for each $i\in[\ell]$,
    $S_{i,n}=T_{i,1}^{a_1(n)}\ldots T_{i,m}^{a_m(n)}$, for some commuting invertible measure-preserving transformations $T_{i,1},\ldots,T_{i,m}\colon X\to X$.
    Suppose that for all $j\in [m]$, $\CP(a_{j})$ is large.
    Then, for any functions $F_1,\ldots, F_\ell \in L^{\infty}(\mu)$ with $|F_i|=1$ for all $i\in[\ell]$, and any $K\in\N$,
    \begin{multline}\label{Katai for actions eq}
        \limsup_{N\to\infty} \Bigg\|\E_{n\in[N]} \prod_{i=1}^\ell S_{i,n}F_i\Bigg\|^2 \\
        \leq \limsup_{N\to\infty} \E_{k_1,k_2\in[K]} \E_{n\in [N]} \int_X \prod_{i\in C_{j_0}} S_{i,n}\left(T_{i,j_0}^{k_1}F_i\cdot T_{i,j_0}^{k_2}\overline{F_i}\right)\, d\mu,
    \end{multline}
    where $C_{j_0}= \{i\in[\ell]\colon T_{i,j_0}\neq I\}$.
\end{Lemma}

\begin{proof} 
This lemma is a simple application of Corollary~\ref{cor_Katai variant}. The sequences in question are functions taking values in the Hilbert space $L^2(\mu)$ and depend only on the nice additive functions $a_1,\ldots, a_m$. This implies that Corollary~\ref{cor_Katai variant} is applicable below.

Without loss of generality, we assume that $j_0=1$. We wish to apply this corollary for the large set $\CP=\CP(a_1)$. Then the left-hand side of \eqref{Katai for actions eq} is bounded by
\begin{equation}\label{bound1}
    \limsup_{M\to\infty} \limsup_{N\to\infty} \E_{k_1,k_2\in[K]} \logE_{\substack{m_1\in \CP^{k_1}(M) \\ m_2\in \CP^{k_2}(M)}}\E_{n\in[N]} \int_X \prod_{i=1}^\ell S_{i,nm_1}F_i\cdot S_{i,nm_2}\overline{F_i}\, d\mu.
\end{equation}
For each $k\in\N$, the set $\CP^k(M)$ consists of $k$-almost primes with prime factors from the set $\CP(a_1)$, where the latter is disjoint from any other set $\CP(a_j)$, $j\neq1$. Thus, for any $m_1\in\CP^{k_1}(M)$ and for any $i\in[\ell]$,
$$S_{i,nm_1} = \prod_{j=1}^m T_{i,j}^{a_j(nm_1)} = T_{i,1}^{a_1(n)+k_1}\prod_{j=2}^m T_{i,j}^{a_j(n)} = T_{i,1}^{k_1} S_{i,n}.$$
Using similar calculations for $\CP^{k_2}(M)$, we can write \eqref{bound1} as
$$\limsup_{N\to\infty} \E_{k_1,k_2\in[K]} \E_{n\in[N]} \int_X \prod_{i=1}^\ell S_{i,n}(T_{i,1}^{k_1}F_i\cdot T_{i,1}^{k_2}\overline{F_i})\, d\mu.$$
Finally, for each $i\in[\ell]\setminus C_1$, we have $T_{i,1}=I$, so $T_{i,1}^{k_1}F_i\cdot T_{i,1}^{k_2}\overline{F_i} = F_i\cdot\overline{F_i} = |F_i|^2 = 1$, and then the result follows.
\end{proof}

\subsection{Averages over $\CP$-free numbers}

For $\CP\subseteq\P$, we define the {\em set of $\CP$-free numbers} $\CQ_\CP$ to be the set of all positive integers whose prime factors lie outside $\CP$, namely,
\begin{equation}\label{P-free_defn}
    \CQ_\CP :=\{n\in\N\colon p\mid n\implies p\not\in\CP\}.
\end{equation}

\begin{Lemma}\label{lemma_ p free numbers density}
    If $\CP\subseteq \P$ is a thin set, then, 
    \begin{equation*}
        d(\CQ_{\CP})=\prod_{p\in \CP} \left(1-\frac{1}{p}\right)>0.
    \end{equation*}
\end{Lemma}
\begin{Remark}
    Note that this lemma still holds for any subsets of primes, with the conclusion being that $d(\CQ_{\CP})=0$ when $\sum_{p\in \CP}\frac{1}{p}=\infty$, i.e., the set $\CP$ is large. 
\end{Remark}

The proof of this lemma is an exercise in \Mobius\ inversion. We omit its proof, as we use similar techniques in the proof of Lemma~\ref{lemma_averages over naturals to averages over P free numbers} below.

The following lemma allows us to freely pass from averages over the interval $[N]$ to averages over the $\CP$-free numbers in this interval whenever we want to prove the convergence of either of the two.

\begin{Lemma}\label{lemma_averages over naturals to averages over P free numbers}
    Let $(\Hilb,\|\cdot\|)$ be a Hilbert space, $\Phi\colon\N_0^m\to\Hilb$ be a $1$-bounded function, $\CP\subseteq\P$ be a thin set, and $a_1,\ldots,a_m\colon\N\to\N_0$ be completely additive functions satisfying, for each $j\in[m]$, $a_j(p)=0$ for all $p\in \CP$.
    If the averages
    $$
        \E_{n\in[N]} \Phi\left(a_1(n),\ldots, a_m(n)\right)
    $$ 
    converge in norm as $N\to\infty$,
    then the averages 
    $$
        \E_{n\in[N]} \1_{\CQ_{\CP}}(n)\Phi\left(a_1(n),\ldots, a_m(n)\right)
    $$
    also converge in norm as $N\to\infty$.\footnote{The converse is also true, but we shall not need it.}
    In addition, if the former averages converge to $L$, then the latter converge to $d(\CQ_\CP)L$.
\end{Lemma}

\begin{proof}
Our argument is a simple application of \Mobius\ inversion, combined with \cite[Lemma 3.7]{Charamaras-multiplicative}. We prove it for $m=1$, as the general case is identical.

First, we show that the conclusion holds for any finite $\CP'\subseteq\CP$. Fix such a $\CP'=\{p_1,\ldots ,p_{r}\}$ and let $W=p_1\ldots p_r$. 
Using the identity $\1_{n=1}=\sum_{d\mid n} \mu(d)$, and the fact that $a_1(dn)=a_1(n)+a_1(d)=a_1(n)$ for every $d\mid W$, we calculate 
\begin{multline*}
    \E_{n\in[N]} \1_{\CQ_{\CP'}}(n) \Phi(a_1(n))
    = \E_{n\in[N]} \1_{(n,W)=1} \Phi(a_1(n))
    = \E_{n\in[N]} \Phi(a_1(n))\sum_{d|(n,W)}\mu(d) \\
    = \sum_{d|W}\mu(d) \Bigg(\frac{1}{N}\sum_{n\leq N/d} \Phi(a_1(dn))\Bigg)
    = \sum_{d\mid W}\frac{\mu(d)}{d}\E_{n\in [N/d]}\Phi(a_1(n)) + \Oh\left(\frac{1}{N}\right).
\end{multline*} 
For each fixed $d$, the averages on the right-hand side converge, and since there are finitely many $d\mid W$, then the right-hand side converges and, using Euler products and Lemma~\ref{lemma_ p free numbers density}, the limit is 
$$
    L\sum_{d\mid W}\frac{\mu(d)}{d}
    = 
    L\prod_{p\in \CP'}\left(1-\frac{1}{p}\right)
    =
    d(\CQ_{\CP'})L.
$$

We can now establish the conclusion of the lemma.
Let $\e>0$. By \cite[Lemma 3.7]{Charamaras-multiplicative}, there is a finite set $\CP_\e\subseteq\CP$ such that $d(\CQ_{\CP_\e}\setminus\CQ_\CP)<\epsilon$.
It follows that
\begin{multline*}
    \Big\|\E_{n\in[N]} \1_{\CQ_\CP}(n) \Phi(a_1(n))-\E_{n\in[N]} \1_{\CQ_{\CP_\e}}(n) \Phi(a_1(n))\Big\| \\        
    \leq \E_{n\in[N]} \1_{\CQ_{\CP_\e}\setminus\CQ_{\CP}}(n)=d(\CQ_{\CP_\e}\setminus \CQ_{\CP})+\oh_{N\to\infty}(1)
    < \e + \oh_{N\to\infty}(1).
\end{multline*}
Using the triangle inequality and the finite case assertion we proved before, we have 
\begin{multline*}
    \Big\|\E_{n\in[N]}\1_{\CQ_\CP}(n) \Phi(a_1(n)) - d(\CQ_\CP)L\Big\| \\
    < \e + \oh_{N\to\infty}(1) + d(\CQ_{\CP_\e}\setminus\CQ_\CP)\|L\|
    < 2\e + \oh_{N\to\infty}(1).
\end{multline*}
Sending $N\to\infty$ and then $\e\to0^+$, we reach the desired conclusion.
\end{proof}

\section{Convergence of multiple averages}
\label{section_proof}

\subsection{Simplifying the average in Theorem~\ref{main theorem}}

Let $(X,\mu,S_1,\ldots,S_\ell)$ be a system. Throughout, we may assume that $\ell\ge 2$, since the $\ell=1$ is known by the Pretentious Mean Ergodic Theorem \cite[Theorem A]{Charamaras-multiplicative}. In view of Lemma~\ref{fg actions characterization lemma}, for each $i\in[\ell]$, we have the following: there exist $m_i\in\N$ commuting invertible measure-preserving transformations $R_{i,1},\ldots,R_{i,m_i}\colon X\to X$ that are pairwise distinct, and some nice finitely generated completely additive functions $b_{i,1},\ldots,b_{i,m_i}\colon\N\to\N_0$, such that for all $n\in\N$,
$$S_{i,n} = R_{i,1}^{b_{i,1}(n)}\ldots R_{i,m_i}^{b_{i,m_i}(n)}.$$
In addition, for each $i$, we define 
$$\CP(b_i) := \bigcup_{j=1}^{m_i}\CP(b_{i,j}).$$

To prove Theorem~\ref{main theorem}, we have to show that for any $F_1,\ldots,F_\ell\in L^\infty(\mu)$, 
\begin{equation}\label{main average general 2}
    \lim_{N\to\infty}\E_{n\in[N]}\prod_{i=1}^\ell\prod_{j=1}^{m_i}R_{i,j}^{b_{i,j}(n)}F_i
\end{equation}
exists in $L^2(\mu)$. 

We proceed to simplify the statement of what we want to show:

(a) We may assume the functions $F_i$, $i\in[\ell]$, satisfy $|F_i|=1$, and, in fact, this can be assumed for any $F\in L^\infty(\mu)$ appearing in the proof. To see why, first observe that we can always assume that $\|F\|_\infty=1$, and then note that any such function can be written as $F=\frac{G+H}{2}$ for two measurable functions $G,H$ with $|G|=|H|=1$.         

(b) All the sets $\CP(b_{i,j})$ can be assumed to be large. To see this, suppose that $\CP=\CP(b_{1,1})$ is thin. 
We write $n=dn'$, where $d$ has all its prime factors in $\CP$ and $n'$ is $\CP$-free, and we split the average over $n$ to an average over $d\in[N]$ with $p\in\CP$ for any prime $p\mid d$, and an average over $n\in[N/d]$ that are $\CP$-free.
For each fixed $d$, we absorb $S_{i,d}$ into $F_i$, and denote this function by $\wt{F}_i$. 
Then the new actions $\wt{S}_i$, defined by 
$\wt{S}_{i,p} = \1_{p\not\in\CP} S_{i,p} + \1_{p\in\CP} I$,
satisfy 
$$
    \wt{S}_{i,n} = \prod_{j=1}^{m_i} R_{i,j}^{\wt{b}_{i,j}(n)},
$$
where each $\wt{b}_{i,j}(p):=\1_{p\not\in\CP}b_{i,j}(p)$ vanishes on $\CP$.
Then the average in \eqref{main average general 2} becomes
\begin{equation}\label{eq_n=dn'}
    \sum_{\substack{d\leq N \\ p\mid d\implies p\in\CP}} \frac{1}{d} \cdot \frac{d}{N}\sum_{\substack{n'\leq N/d \\ p\mid n'\implies p\not\in\CP}} \prod_{i=1}^\ell \prod_{j=1}^{m_i} R_{i,j}^{\wt{b}_{i,j}(n')}\wt{F}_i.
\end{equation}
Since, for each $d$, the $L^2(\mu)$ norm of the sequence
$$
    \frac{\1_{p\mid d\implies p\in\CP}}{d} \cdot \frac{d}{N}\sum_{\substack{n'\leq N/d \\ p\mid n'\implies p\not\in\CP}} \prod_{i=1}^\ell \prod_{j=1}^{m_i} R_{i,j}^{\wt{b}_{i,j}(n')}\wt{F}_i
$$
is bounded by 
$$
    \frac{\1_{p\mid d\implies p\in\CP}}{d},
    \quad \text{and} \quad
    \sum_{\substack{d\geq1 \\ p\mid d\implies p\in\CP}}\frac{1}{d} 
    = \prod_{p\in\CP}\bigg(1-\frac{1}{p}\bigg)^{-1}
    < \infty,
$$
it follows by the dominated convergence theorem that norm convergence of \eqref{eq_n=dn'} follows by norm convergence of the above sequence for each $d$. 
It, therefore, suffices to show norm convergence for
$$
    \frac{1}{N}\sum_{\substack{n'\leq N \\ p\mid n'\implies p\not\in\CP}} \prod_{i=1}^\ell \prod_{j=1}^{m_i} R_{i,j}^{\wt{b}_{i,j}(n')}\wt{F}_i.
$$
By Lemma~\ref{lemma_averages over naturals to averages over P free numbers}, it suffices to prove norm convergence for
$$
    \E_{n\in[N]}\prod_{i=1}^\ell \prod_{j=1}^{m_i} R_{i,j}^{\wt{b}_{i,j}(n)}\wt{F}_i.
$$
Noting that $\CP(\wt{b}_{1,1})=\emptyset$, we have $R_{1,1}^{\wt{b}_{1,1}(n)}=I$ for all $n$, so it vanishes from the average. 

We repeat this for another pair $(i,j)\neq(1,1)$ for which $\CP(\wt{b}_{i,j})$ is thin, and continue in this way, until we exhaust all the (finitely many) pairs $(i,j)$ for which the corresponding sequence has thin prime support. In this way, we are reduced to proving convergence of an average where all the completely additive functions involved have large prime support.

(c) Finally, we have the following.

\begin{Claim}\label{claim}
    To prove convergence, it suffices to consider actions of the form
    \begin{equation}\label{actions final form eq}
        S_{i,n}=\prod_{j=1}^m T_{i,j}^{a_j(n)},
    \end{equation}
    where $m\in\N$, the functions $a_1,\ldots,a_m\colon\N\to\N_0$ are nice finitely generated completely additive functions with $\CP(a_j)$ large for all $j\in[m]$, and the $T_{i,j}$ are commuting invertible measure-preserving transformations (allowing $T_{i,j}=I$ and $T_{i,j}=T_{i,j'}$ for distinct $j,j'$).
    Then, we identify each action $S_i$ with $(T_{i,1},\ldots,T_{i,m})$, we write $S_i=(T_{i,1},\ldots,T_{i,m})$, and we say that {\em the transformation $T_{i,j}$ is in the $a_j$-position of the action $S_i$}. When we want to specify the functions $a_1,\ldots,a_m$, we say $S_i=(T_{i,1},\ldots,T_{i,m})$ {\em relative to} $a_1,\ldots,a_m$.  
\end{Claim}
         
\begin{proof}
Consider all the distinct non-empty sets of the form
$$\left(\CP(b_{i_1,j_1})\cap\ldots\cap\CP(b_{i_s,j_s})\right)\setminus\bigcup_{i\in[\ell]\setminus\{i_1,\ldots,i_s\}}\CP(b_i),$$
where $s\in[\ell]$, $i_1,\ldots,i_s\in[\ell]$ are pairwise distinct, and $j_1\in[m_{i_1}],\ldots,j_s\in[m_{i_s}]$.
Let $m\in\N$ be the number of all these sets and consider an enumeration $\{\CP_1,\ldots,\CP_m\}$. Since $\CP(b_{i,j})\cap\CP(b_{i,j'})=\emptyset$ for all $i\in[\ell]$ and all distinct $j,j'\in[m_i]$, then any two distinct $\CP_j,\CP_{j'}$ are disjoint. For each $j\in[m]$, we define the finitely generated completely additive function
$a_j\colon\N\to\N_0$ by $a_j(p)=\1_{\CP_j}(p)$. We note that the functions $a_1,\ldots,a_m$ are nice, since for each $j\neq j'\in[m]$ we have $\CP(a_j)\cap\CP(a_{j'})=\CP_j\cap\CP_{j'}=\emptyset$. Recalling also that $b_{i,j}(p)=\1_{\CP(b_{i,j})}(p)$ for all $i\in[\ell]$ and all $j\in[m_i]$, it follows that for each $i,j$, there is $(\e_{i,j}(k))_{k=1}^m\in\{0,1\}^m$, defined by $\e_{i,j}(k)=\1_{\CP_k\subseteq\CP(b_{i,j})}$, such that
$$b_{i,j} = \sum_{k=1}^m \e_{i,j}(k)a_k.$$
In addition, by the definition of the sets $\CP_k$, for each $i,j$, there is a finite collection of these sets that partition $\CP(b_{i,j})$ (and these are exactly the sets $\CP_k$ with $\e_{i,j}(k)=1$), so it follows that
$$R_{i,j}^{b_{i,j}(n)} = \prod_{k=1}^m\left(R_{i,j}^{\e_{i,j}(k)}\right)^{a_k(n)}.$$
Thus, 
$$\prod_{j=1}^{m_i}R_{i,j}^{b_{i,j}(n)} = \prod_{j=1}^{m_i}\prod_{k=1}^m\left(R_{i,j}^{\e_{i,j}(k)}\right)^{a_k(n)}
=\prod_{k=1}^m\Bigg(\prod_{j=1}^{m_i}R_{i,j}^{\e_{i,j}(k)}\Bigg)^{a_k(n)}.$$
For each $i\in[\ell]$ and $k\in[m]$, define
$T_{i,k}=\prod_{j=1}^{m_i}R_{i,j}^{\e_{i,j}(k)}$.
This gives the asserted representation, except that some of the $\CP(a_k)$ may be thin. Applying the reduction in (b) to each such support removes the corresponding function $a_k$.
Since these supports are pairwise disjoint, the remaining functions are unchanged.
Relabeling the reduced actions and the remaining functions proves the claim.
\end{proof}

\begin{Definition}[Collection of actions]\label{collection of actions}
    Let $(X,\mu,S_1,\ldots,S_\ell)$ be a system such that for all $i\in[\ell]$, $S_i = (T_{i,1},\ldots,T_{i,m})$.
    We call
    \begin{equation}\label{set_of_actions}
        \CS=\{S_i\colon i\in[\ell]\}
    \end{equation}
    a {\em collection of actions (on the probability space $\xm$) of length $\ell(\CS)=\ell$}.

    We say that the collection $\CS$ is {\em nice} if it satisfies the following four conditions:
    \begin{enumerate}[label=(\arabic*)]
        \item\label{condition1} For all $j\in[m]$, there are $i\neq i'\in[\ell]$ such that $T_{i,j}\neq T_{i',j}$.
        \item\label{condition2} For all $i\in[\ell]$ and all $j\in\{2,\ldots,m\}$, $T_{i,j}\neq I$.
        \item\label{condition3} For all $i,i'\in[\ell]$, $j\neq j'\in[m]$, we have $T_{i,j}\neq T_{i',j'}$.
        \item\label{condition4} The set $\{T_{i,j}\neq I\colon i\in[\ell], j\in[m]\}$ freely generates an abelian group.
    \end{enumerate}
\end{Definition}

Let $\CS=\{S_i\colon i\in[\ell]\}$ be a collection of actions on some probability space $\xm$ such that for each $i\in[\ell]$, $S_i=(T_{i,1},\ldots,T_{i,m})$ relative to some $a_1,\ldots,a_m$.
Given any collection $\CF$ of functions $F_i\in L^\infty(\mu)$, $i\in[\ell]$, with $|F_i|=1$, we let
\begin{equation}\label{main_average}
    A_N(\CS;\CF) := \E_{n\in[N]}\prod_{i=1}^\ell S_{i,n}F_i
    = \E_{n\in[N]}\prod_{i=1}^\ell \prod_{j=1}^mT_{i,j}^{a_j(n)} F_i.
\end{equation}
Our goal is to show that $A_N(\CS;\CF)$
converges in $L^2(\mu)$ for all such $\CF$. Whenever this is the case, we say that we have {\em convergence for} $\CS$. The next lemma allows to assume that $\CS$ is a nice collection.

\begin{Lemma}
    If we have convergence for all nice collections of actions, then we have convergence for all collections of actions.
\end{Lemma}

\begin{proof}
Let $\CS=\{S_i\colon i\in[\ell]\}$ be a collection of actions on some probability space $\xm$ such that for each $i\in[\ell]$, $S_i=(T_{i,1},\ldots,T_{i,m})$ relative to some $a_1,\ldots,a_m$, and $\CF=\{F_i\in L^\infty(\mu)\colon |F_i|=1,i\in[\ell]\}$.
To prove the lemma, it suffices to construct a probability space $(\wt{X},\wt{\mu})$, a nice collection of actions $\wt{\CS}=\{\wt{S}_i\colon i\in[\ell]\}$ on $(\wt{X},\wt{\mu})$ with $\wt{S}_i=(\wt{T}_{i,1},\ldots,\wt{T}_{i,m})$ relative to $a_1,\ldots,a_m$, and $\wt{\CF}=\{\wt{F}_i\in L^\infty(\wt{\mu})\colon |\wt{F}_i|=1,i\in[\ell]\}$, such that $(\wt{X},\wt{\mu},\wt{S})$ is an extension of $\xms$ and
\begin{equation}\label{eq_equal_functions}
    \wt{S}_{i,n}\wt{F}_i = (S_{i,n}F_i)\circ\pi, \quad \forall i\in[\ell],n\in\N,
\end{equation}
where $\pi\colon\wt{X}\to X$ is the factor map.

We let $k=\ell m$. We define $\wt{X} = X\times\T^k$, $\wt{\mu} = \mu\times m_{\T^k}$, where $m_{\T^k}$ is the Haar measure on the $k$ dimensional torus. We fix an irrational number $\alpha$ and an enumeration map $u\colon[\ell]\times[m]\to[k]$. For each $r\in[k]$, we consider the measure-preserving transformation $R_r$ on $(\T^k,m_{\T^k})$, given by
$$
R_r(y_1,\ldots,y_k) = (y_1,\ldots,y_{r-1},y_r+\alpha,y_{r+1},\ldots,y_k).
$$
We define, for each $i,j$, the measure-preserving transformation $\wt{T}_{i,j} = T_{i,j}\times R_{u(i,j)}$ on $(\wt{X},\wt{\mu})$, $\wt{S}_i = (\wt{T}_{i,1},\ldots,\wt{T}_{i,m})$ relative to $a_1,\ldots,a_m$, and we consider the collection $\wt{\CS} = \{\wt{S}_i\colon i\in[\ell]\}$. Let $\pi\colon\wt{X}\to X$ be the projection $\pi(x,y)=x$.
Finally, consider, for each $i$, the function $\wt{F}_i\in L^\infty(\wt{\mu})$ given by $\wt{F}_i(x,y) = F_i(x)$. Equivalently, $\wt{F}_i=F_i\circ\pi$. By construction, we have 
$$\pi\circ\wt{T}_{i,j} = T_{i,j}\circ\pi, \quad \forall i\in[\ell],j\in[m],$$
from which \eqref{eq_equal_functions} clearly follows.

It is left to show that the new collection $\wt{\CS}$ is nice.
In fact, it is enough to establish that the transformations $R_{i,j}:=R_{u(i,j)}$, for $i\in[\ell],j\in[m]$, satisfy the conditions \ref{condition1}--\ref{condition4}.

For \ref{condition1}, if $i\neq i'$, then $u(i,j)\neq u(i',j)$, and hence $R_{i,j}\neq R_{i',j}$. For \ref{condition2}, we simply use that $\alpha$ is irrational, so $R_{i,j}\neq I$ for all $i,j$. For \ref{condition3}, if $j\neq j'$, then $u(i,j)\neq u(i',j')$ for all $i,i'\in[\ell]$, and hence $R_{i,j}\neq R_{i',j'}$. Finally, \ref{condition4} follows from the fact that for any $k$ the transformations $R_1,\ldots,R_k$ generate a free abelian group. To see this, if $R_1^{q_1}\cdots R_k^{q_k}=I$
for some $q_1,\ldots,q_k\in\mathbb Z$, then $q_r\alpha\in\Z$ for every $r\in[k]$, and hence $q_r=0$ for every $r$, since $\alpha$ is irrational. This completes the proof of the lemma.
\end{proof}

Thus far, we have reduced the proof of \cref{main theorem} to the following.

\begin{Theorem}\label{main_theorem_simple}
    Let $\CS=\{S_i\colon i\in[\ell]\}$ be a nice collection of actions on some probability space $\xm$ such that for each $i\in[\ell]$, $S_i = (T_{i,1},\ldots,T_{i,m})$ relative to some nice $a_1,\ldots,a_m$ with large prime supports. Then we have convergence for $\CS$.
\end{Theorem}

The goal of the remainder of this section is to prove Theorem~\ref{main_theorem_simple}.

\subsection{Example}

Before we delve into the details of the proof, we illustrate our methods with a motivating example. 

Suppose we have a probability space $(X,\mu)$, invertible commuting measure-preserving transformations $T_1,T_2,R_1,R_2$, and completely additive functions $a_1,a_2:\N\to \N_0$ that take values in $\{0,1\}$ on primes. Furthermore, we assume that $\CP(a_1)\cap \CP(a_2)=\emptyset$ and that both $\CP(a_1)$ and $\CP(a_2)$
are large.
We will study the convergence of the ergodic averages
\begin{equation}\label{eq_example}
    \E_{n\in [N]}T_1^{a_1(n)}T_2^{a_2(n)}F\cdot R_1^{a_1(n)}R_2^{a_2(n)}G \cdot R_1^{a_1(n)}T_2^{a_2(n)}H\cdot  R_2^{a_2(n)}W
\end{equation}for any functions $F,G,H,W\in L^{\infty}(\mu).$
Without loss of generality, we will always assume that all measurable functions involved have unit modulus.

$\bullet$ {\bf First reduction:}
We apply Corollary~\ref{cor_Katai variant} for the set $\CP=\CP(a_1)$ to get the bound 
\begin{align*}
    \limsup_{N\to\infty} & \norm{\E_{n\in [N]}T_1^{a_1(n)}T_2^{a_2(n)}F\cdot R_1^{a_1(n)}R_2^{a_2(n)}G \cdot R_1^{a_1(n)}T_2^{a_2(n)}H\cdot R_2^{a_2(n)}W}_2^2 \\
    & \leq \limsup_{M\to\infty}\limsup_{N\to\infty}\E_{k_1,k_2\in[K]}\logE_{\substack{m_1\in\CP^{k_1}(M) \\ m_2\in\CP^{k_2}(M)}} \E_{n\in[N]}\int_X \overline{T_1^{a_1(m_1n)}T_2^{a_2(m_1n)}F} \\
    & \hspace{.5cm} \cdot \overline{R_1^{a_1(m_1n)}R_2^{a_2(m_1n)}G\cdot R_1^{a_1(m_1n)}T_2^{a_2(m_1n)}H\cdot  R_2^{a_2(m_1n)}W}\cdot
    T_1^{a_1(m_2n)}T_2^{a_2(m_2n)}F\\
    & \hspace{.5cm} \cdot R_1^{a_1(m_2n)}R_2^{a_2(m_2n)}G \cdot R_1^{a_1(m_2n)}T_2^{a_2(m_2n)}H\cdot  R_2^{a_2(m_2n)}W \,d\mu
\end{align*}
for any $K\in\N$. 

Now, we observe that $a_1(m_1n)=a_1(m_1)+a_1(n)=k_1+a_1(n)$, since $m_1$ is a product of $k_1$ primes from $\CP(a_1)$. Similarly, $a_1(m_2n)=k_2+a_1(n)$.
On the contrary, we have $a_2(m_1n)=a_2(m_2n)=a_2(n)$, since $a_2(p)=0$ for every prime $p\in \CP(a_1)$ by our assumption. It follows that the previous bound simplifies to 
\begin{align*} 
    & \limsup_{N\to\infty}\norm{\E_{n\in [N]}T_1^{a_1(n)}T_2^{a_2(n)}F\cdot R_1^{a_1(n)}R_2^{a_2(n)}G \cdot R_1^{a_1(n)}T_2^{a_2(n)}H\cdot R_2^{a_2(n)}W }_2^2 \\
    & \leq \limsup_{K\to\infty}\E_{k_1,k_2\in [K]}\limsup_{N\to\infty} \E_{n\in[N]}\\
    &\qquad  \int_X \overline{T_1^{a_1(n)+k_1}T_2^{a_2(n)}F\cdot R_1^{a_1(n)+k_1}R_2^{a_2(n)}G} \cdot\overline{R_1^{a_1(n)+k_1}T_2^{a_2(n)}H\cdot  R_2^{a_2(n)}W} \\
    & \hspace{2.9cm} \cdot
    T_1^{a_1(n)+k_2}T_2^{a_2(n)}F\cdot R_1^{a_1(n)+k_2}R_2^{a_2(n)}G \cdot R_1^{a_1(n)+k_2}T_2^{a_2(n)}H\cdot  R_2^{a_2(n)}W \,d\mu \\
    & = \limsup_{K\to\infty}\E_{k_1,k_2\in [K]}\limsup_{N\to\infty} \E_{n\in[N]}\\
    &\qquad \int_X T_1^{a_1(n)}T_2^{a_2(n)}\big(T_1^{k_1}\overline{F}\cdot T_1^{k_2}F\big) \cdot R_1^{a_1(n)}R_2^{a_2(n)}\big(R_1^{k_1}\overline{G}\cdot R_1^{k_2}G\big) \\
    & \hspace{2.9cm}\qquad \qquad \qquad \qquad \qquad \ \ \ \cdot R_1^{a_1(n)}T_2^{a_2(n)}\big(R_1^{k_1}\overline{H}\cdot R_1^{k_2}H\big)\cdot R_2^{a_2(n)}|W|^2\,d\mu.
\end{align*} 
Composing with $R_1^{-a_1(n)}T_2^{-a_2(n)}$ and using the fact that $|W|=1$ almost everywhere, we conclude that 
\begin{align*}
     & \limsup_{N\to\infty}\norm{\E_{n\in [N]}T_1^{a_1(n)}T_2^{a_2(n)}F\cdot R_1^{a_1(n)}R_2^{a_2(n)}G \cdot R_1^{a_1(n)}T_2^{a_2(n)}H\cdot R_2^{a_2(n)}W}_2^2 \\
     & \leq \E_{k_1,k_2\in [K]} \E_{n\in[N]}\int_X (T_1R_1^{-1})^{a_1(n)}\big(T_1^{k_1}\overline{F}\cdot T_1^{k_2}F\big)\cdot (R_2T_2^{-1})^{a_2(n)}\big(R_1^{k_1}\overline{G}\cdot R_1^{k_2}G\big) \\
     & \hspace{4cm} \cdot \big(R_1^{k_1}\overline{H}\cdot R_1^{k_2}H\big) \,d\mu. 
\end{align*}

Averaging over $k_1,k_2$ and applying the Cauchy-Schwarz inequality, we deduce that 
\begin{align*}
    & \limsup_{N\to\infty}\norm{\E_{n\in [N]}T_1^{a_1(n)}T_2^{a_2(n)}F\cdot R_1^{a_1(n)}R_2^{a_2(n)}G \cdot R_1^{a_1(n)}T_2^{a_2(n)}H\cdot  R_2^{a_2(n)}W }_2^2 \\
    & \leq \E_{k_1,k_2\in [K]}\limsup_{N\to\infty}\norm{\E_{n\in [N]}(T_1R_1^{-1})^{a_1(n)}\big(T_1^{k_1}\overline{F}\cdot T_1^{k_2}F\big)\cdot (R_2T_2^{-1})^{a_2(n)}\big(R_1^{k_1}\overline{G}\cdot R_1^{k_2}G\big)}_2.
\end{align*}
We apply Corollary~\ref{cor_Katai variant} for $\CP=\CP(a_1)$ once more to bound the last norm.
We infer that  
$$
\begin{aligned}
&\limsup_{N\to\infty}\norm{\E_{n\in[N]}(T_1R_1^{-1})^{a_1(n)}\big(T_1^{k_1}\overline{F}\cdot T_1^{k_2}F\big)\cdot(R_2T_2^{-1})^{a_2(n)}\big(R_1^{k_1}\overline{G}\cdot R_1^{k_2}G\big)}_2^2\\
&\leq\lim_{K\to\infty}\lim_{M\to\infty}\lim_{N\to\infty}\E_{\ell_1,\ell_2\in[K]}\logE_{\substack{m_1\in\CP^{\ell_1}(M)\\m_2\in\CP^{\ell_2}(M)}}\E_{n\in[N]}\\
&\qquad\int_X\overline{(T_1R_1^{-1})^{a_1(m_1n)}\big(T_1^{k_1}\overline{F}\cdot T_1^{k_2}F\big)}\cdot\overline{(R_2T_2^{-1})^{a_2(m_1n)}\big(R_1^{k_1}\overline{G}\cdot R_1^{k_2}G\big)}\\
&\qquad \qquad \cdot(T_1R_1^{-1})^{a_1(m_2n)}\big(T_1^{k_1}\overline{F}\cdot T_1^{k_2}F\big)
\cdot(R_2T_2^{-1})^{a_2(m_2n)}\big(R_1^{k_1}\overline{G}\cdot R_1^{k_2}G\big)\,d\mu.
\end{aligned}
$$
Similarly as above, we compute $a_1(m_1n)=\ell_1+a_1(n)$, $a_1(m_2n)=\ell_2+a_1(n)$ and $a_2(m_1n)=a_2(m_2n)=a_2(n)$.
Plugging these relations into the previous expression, we deduce that 
\begin{align*}
    & \limsup_{N\to\infty}\norm{\E_{n\in [N]}(T_1R_1^{-1})^{a_1(n)}\big(T_1^{k_1}\overline{F}\cdot T_1^{k_2}F\big)\cdot (R_2T_2^{-1})^{a_2(n)}\big(R_1^{k_1}\overline{G}\cdot R_1^{k_2}G\big)}_2^2 \\
    & \hspace{3cm} \leq \lim_{N\to\infty}\E_{\ell_1,\ell_2\in [K]}\logE_{\substack{m_1\in\CP^{\ell_1}(M) \\ m_2\in\CP^{\ell_2}(M)}} \E_{n\in[N]}\int_X \overline{(T_1R_1^{-1})^{a_1(n)+\ell_1}\big(T_1^{k_1}\overline{F}\cdot T_1^{k_2}F\big)} \\
    & \hspace{3.5cm} \cdot (T_1R_1^{-1})^{a_1(n)+\ell_2}\big(T_1^{k_1}\overline{F}\cdot T_1^{k_2}F\big)\cdot(R_2T_2^{-1})^{a_2(n)}\big|R_1^{k_1}\overline{G}\cdot R_1^{k_2}G\big|^2\,d\mu.
\end{align*} 
Using that $|G|=1$ almost everywhere, composing with $(T_1R_1^{-1})^{-a_1(n)}$ and then averaging over $\ell_1,\ell_2$, we get the final bound
\begin{align*}
    &\limsup_{N\to\infty}\norm{\E_{n\in[N]}T_1^{a_1(n)}T_2^{a_2(n)}F\cdot R_1^{a_1(n)}R_2^{a_2(n)}G\cdot R_1^{a_1(n)}T_2^{a_2(n)}H\cdot R_2^{a_2(n)}W}_2^4 \\
    & \hspace{3.8cm} \leq \E_{\ell_1,\ell_2,k_1,k_2\in[K]} \int_X T_1^{k_1}(T_1R_1^{-1})^{\ell_1}F\cdot T_1^{k_1}(T_1R_1^{-1})^{\ell_2}\overline{F}\\
    & \hspace{7cm} \cdot T_1^{k_2}(T_1R_1^{-1})^{\ell_1}\overline{F}\cdot T_1^{k_2}(T_1R_1^{-1})^{\ell_2}F\,d\mu .
\end{align*}
Taking the limit $K\to\infty$, we have the seminorm control \begin{multline*}
      \limsup_{N\to\infty}\norm{\E_{n\in [N]}T_1^{a_1(n)}T_2^{a_2(n)}F\cdot R_1^{a_1(n)}R_2^{a_2(n)}G \cdot R_1^{a_1(n)}T_2^{a_2(n)}H\cdot  R_2^{a_2(n)}W }_2 \\
      \leq \nnorm{F}_{T_1,T_1R_1^{-1}}.
\end{multline*}

This seminorm bound holds in any system $(X,\mu)$ equipped with commuting measure-preserving transformations $T_1,T_2,R_1,R_2$. In order to extract useful information from this seminorm, it suffices to consider the case where the system is magic with respect to $T_1,T_1R_1^{-1}$. Namely, we may pass to an extension of the original system that is magic with respect to these two transformations and it suffices to prove that averages of the form \eqref{eq_example} converge in this extension.

If the function $F$ satisfies $\nnorm{F}_{T_1,T_1R_1^{-1}}=0$, then the corresponding average converges to zero and we are done. Assume, therefore, that $F$ is measurable with respect to $\mathcal{Z}_{T_1,T_1R_1^{-1}}$. Since the linear span of functions of the form $F_1\cdot F_2$ with $F_1\in \mathcal{I}(T_1), F_2\in \mathcal{I}(T_1R_1^{-1})$ is dense in $\mathcal{Z}_{T_1,T_1R_1^{-1}}$ (due to the magic system assumption), it suffices to prove our theorem in the case that $F$ has this form. In this case, our average can be rewritten as
\begin{multline*}
    \E_{n\in [N]}T_1^{a_1(n)}T_2^{a_2(n)}(F_1\cdot F_2)\cdot R_1^{a_1(n)}R_2^{a_2(n)}G \cdot R_1^{a_1(n)}T_2^{a_2(n)}H\cdot  R_2^{a_2(n)}W \\
    = \E_{n\in [N]}T_2^{a_2(n)}F_1 \cdot R_1^{a_1(n)}R_2^{a_2(n)}G \cdot R_1^{a_1(n)}T_2^{a_2(n)}(H\cdot F_2)\cdot  R_2^{a_2(n)}W.
\end{multline*}
Thus, our result will follow if we show that for any system $(X,\mu)$, measure-preserving transformation $T_1,T_2,R_1,R_2$ and functions $F,G,H,W\in L^{\infty}(\mu)$, the averages \begin{equation}\label{eq_example first reduction}
    \E_{n\in [N]}R_1^{a_1(n)}R_2^{a_2(n)}F \cdot R_1^{a_1(n)}T_2^{a_2(n)}G \cdot T_2^{a_2(n)}H\cdot R_2^{a_2(n)}W
\end{equation}
converge in $L^2(\mu)$.

$\bullet$ {\bf Second reduction:}
We establish the convergence of the last average by adopting the same strategy as above to reduce our problem to the convergence of an average of lower complexity. Again, we can assume that the functions $F,G,H,W$ have unit modulus.

We apply Corollary~\ref{cor_Katai variant} for $\CP(a_1)$ to bound the $L^2$-norm of this average by 
\begin{equation*}
   \limsup_{N\to\infty} \norm{\E_{n\in [N]}(R_2T_2^{-1})^{a_2(n)}\big(R_1^{k_1}\overline{F}\cdot R_1^{k_2}F\big)}_2
\end{equation*}
for any $k_1,k_2\in \N$.
Once again, we apply Corollary~\ref{cor_Katai variant}, this time for $\CP(a_2)$, to bound the fourth power of $L^2$-norm of the initial average by \begin{equation*}
    \E_{k_1,k_2,\ell_1,\ell_2\in [K]}\int_X R_1^{k_1}(R_2T_2^{-1})^{\ell_1}F \cdot  R_1^{k_1}(R_2T_2^{-1})^{\ell_2}\overline{F} \cdot  R_1^{k_2}(R_2T_2^{-1})^{\ell_1}\overline{F}\cdot  R_1^{k_2}(R_2T_2^{-1})^{\ell_2}F\, d\mu.
\end{equation*} 
Taking the limit as $K\to \infty$, we infer \begin{equation*}
    \limsup_{N
    \to\infty} \norm{\E_{n\in [N]}R_1^{a_1(n)}R_2^{a_2(n)}F \cdot R_1^{a_1(n)}T_2^{a_2(n)}G \cdot T_2^{a_2(n)}H\cdot R_2^{a_2(n)}W}_2\leq \nnorm{F}_{R_1,R_2T_2^{-1}}.
\end{equation*}
Once more, we can pass to a magic extension (with respect to $R_1,R_2T_2^{-1}$) of the system and it suffices to prove the convergence of the averages \eqref{eq_example first reduction}. Therefore, we may assume that our system is magic with respect to $R_1,R_2T_2^{-1}$. The seminorm bound implies that it suffices to verify the asserted convergence when $F\in \mathcal{Z}_{R_1,R_2T_2^{-1}}$ and, arguing as above, we can further assume that $F$ takes the form $F=F_1\cdot F_2$, where $F_1\in \mathcal{I}(R_1), F_2\in \mathcal{I}(R_2T_2^{-1})$. Thus, our average can be rewritten as 
\begin{multline*}
    \E_{n\in [N]}R_1^{a_1(n)}R_2^{a_2(n)}(F_1\cdot F_2) \cdot R_1^{a_1(n)}T_2^{a_2(n)}G \cdot T_2^{a_2(n)}H\cdot R_2^{a_2(n)}W \\
    = \E_{n\in [N]} R_1^{a_1(n)}T_2^{a_2(n)}(G\cdot F_2)\cdot T_2^{a_2(n)}H\cdot R_2^{a_2(n)}(W\cdot F_1).
\end{multline*}
Consequently, our result will follow if we show that for any system $(X,\mu)$, measure-preserving transformations $T_1,T_2,R_1,R_2$ and functions $F,G,H\in L^{\infty}(\mu)$, the averages \begin{equation}\label{eq_example second reduction}
     \E_{n\in [N]} R_1^{a_1(n)}T_2^{a_2(n)}F \cdot T_2^{a_2(n)}G\cdot R_2^{a_2(n)}H
\end{equation}
converge in $L^2(\mu)$.

$\bullet$ {\bf Third reduction:}
Again, assume that the functions $F,G,H$ have unit modulus and apply Corollary~\ref{cor_Katai variant} for $\CP(a_1)$ to bound our average by \begin{equation*}
    \E_{k_1,k_2\in[K]} \int_X R_1^{k_1}\overline{F}\cdot R_1^{k_2}F\,d\mu
\end{equation*}for all $k_1,k_2\in \N$.
Using the mean ergodic theorem, we arrive at the bound \begin{equation*}
      \limsup_{N\to\infty}\norm{\E_{n\in [N]} R_1^{a_1(n)}T_2^{a_2(n)}F \cdot T_2^{a_2(n)}G\cdot R_2^{a_2(n)}H}_2\leq \norm{\mathcal{E}_{\mu}(F| \mathcal{I}(R_1))}_2.
\end{equation*}
If $F$ is orthogonal to the invariant factor corresponding to $R_1$, the average converges to zero. Therefore, it suffices to consider the case where $F\in \mathcal{I}(R_1)$, in which case the average simplifies to \begin{equation*}
    \E_{n\in [N]}T_2^{a_2(n)}(G\cdot F)\cdot R_2^{a_2(n)}H. 
\end{equation*} Thus, our result follows if we show that for any system $(X,\mu)$, commuting measure-preserving transformations $T_2,R_2$ and functions $F,G\in L^{\infty}(\mu)$, the averages 
\begin{equation}\label{eq_example third reduction}
     \E_{n\in [N]}T_2^{a_2(n)}F \cdot R_2^{a_2(n)}G
\end{equation}
converge in $L^2(\mu)$. 

$\bullet$ {\bf Final step:}
In order to get the final seminorm bound, we apply Corollary~\ref{cor_Katai variant} for $\CP(a_2)$ twice as we did in the previous steps. After similar computations, we arrive at \begin{equation*}
    \limsup_{N\to\infty} \norm{\E_{n\in [N]} T_2^{a_2(n)}F \cdot R_2^{a_2(n)}G}_2\leq \nnorm{F}_{T_2,T_2R_2^{-1}}.
\end{equation*}
Passing to an appropriate extension, we may assume that our system is magic with respect to the transformations $T_2,T_2R_2^{-1}$. The last seminorm bound implies that it suffices to establish convergence when $F\in \mathcal{Z}_{T_2,T_2R_2^{-1}}$. Using the magic property and an approximation argument, we may reduce to the case where $F=F_1\cdot F_2$, where $F_1\in \mathcal{I}(T_2), F_2\in \mathcal{I}(T_2R_2^{-1})$. Plugging this in \eqref{eq_example third reduction}, we arrive at the average 
\begin{equation}\label{eq_example final step}
    F_1\cdot  \E_{n\in [N]} R_2^{a_2(n)}(F_2\cdot G).
\end{equation}
However, this average involves only a single multiplicative action, and, therefore, it converges in $L^2(\mu)$ by Theorem~\ref{thm_PMET}. The conclusion follows.

\subsection{The complexity vector}

In the above example, the main idea behind proving convergence is to use our orthogonality criterion to bound the averages by a box seminorm. Then, by invoking the machinery of magic extensions, we reduce the problem to establishing convergence for simpler averages. Thus, we repeatedly replace the original average with progressively simpler ones.
This example illustrates the general strategy of the proof. We formalize it by introducing a suitable notion of complexity and a reduction procedure that decreases this complexity after finitely many steps, eventually leading to an average whose convergence is already known. 
The definition of complexity is given below, and the procedure through which the complexity is reduced is described in the following subsection.

\begin{Definition}[Complexity vector]\label{complexity vector defn}
    Let $\CS=\{S_i\colon i\in[\ell]\}$ be a collection of actions on some probability space $\xm$ such that for each $i\in[\ell]$, $S_i = (T_{i,1},\ldots,T_{i,m})$ relative to some $a_1,\ldots,a_m$.
    We define the {\em complexity vector} $\mathbf{c}=\mathbf{c}(\CS)=(c_1(\CS),c_2(\CS),c_3(\CS))=(c_1,c_2,c_3)$ of $\CS$ as follows:
    For each $i\in[\ell]$ and $j\in[m]$, we define the following:
    \begin{itemize}
        \item[i)] If $T_{i,j}\neq I$, then let 
            $$c(T_{i,j}) = c(T_{i,j};\CS) = |\{i'\in[\ell]\colon T_{i',j}=T_{i,j}\}|.$$
        
        \item[ii)]
            $$c(a_j)= c(a_j;\CS) = |\{T_{i,j}\neq I\colon T_{i,j}\ \text{is in the}\ a_j\text{-position of}\ S_i\ \text{for some}\ i\in[\ell]\}|.$$
    \end{itemize}

If $\ell(\CS)=1$, we set $\mathbf{c}(\CS)=(1,1,1)$ by convention.
For $\ell(\CS)\geq2$, we define the complexity vector as follows.
    By reordering, we may assume that there is $0\le r\le m-1$\footnote{Note that $r$ could be equal to $m$, but this corresponds to a singleton collection for which convergence is known by Theorem~\ref{thm_PMET}.} such that for any $j\in[r]$, $T_{i,j}=T_{i',j}$ for all $i,i'\in[\ell]$, and for all $r<j\leq m$, there are $i\neq i'$ such that $T_{i,j}\neq T_{i',j}$.
    Then, we let
    \begin{align*}
        c_1 & = m-r, \\
        c_2 & =
        \min\{c(a_j)\colon r+1\leq j\leq m\}, \\
        c_3 & = \min\big\{c(T_{i,j})\colon T_{i,j}\neq I, i\in[\ell], j=\min\{r+1\leq k\leq m\colon c(a_k)=c_2\}\big\}.
    \end{align*}
    We call $c_1,c_2$ and $c_3$ the {\em first}, {\em second} and {\em third complexity parameter} respectively.
    
    Finally, we denote by $\mathbf{C}$ the set of all $\mathbf{c}\in\N^3$ such that $\mathbf{c}$ is a complexity vector for a collection $\CS$.
\end{Definition}

We remark that if $\CS$ satisfies \ref{condition1}, then $r=0$, thus, $c_1=m$.

Note that by rearranging, if necessary, we may assume throughout that $c_2=c(a_1)$.
We equip $\mathbf{C}$ with the lexicographic order and denote its associated strict order by $<$. 

Let us now explain how the complexity is reduced in the example from the previous subsection. The initial average \eqref{eq_example} has complexity vector $(2,2,1)$. Here, $c_1=2$ is the number of positions (generators) of each action, and $c_2=2$ is the minimum number of nonidentity transformations appearing in some position. Since $c(a_1)=c(a_2)=2$, then $c_2=c(a_1)=2$, and so $c_3=1$ is the minimum multiplicity among the transformations in the $a_1$-position.
More concretely, in the $a_1$-position there are two nonidentity transformations, namely $T_1$ and $R_1$, so this position realizes the value $c_2=2$. Moreover, $T_1$ appears only once, whereas $R_1$ appears twice, and hence $c_3=1$ is attained by $T_1$. The reduction strategy is therefore to replace $T_1$ by $I$ and $R_1$, thereby decreasing the complexity to $(2,1,c_3')$.
Indeed, after the first reduction we are left with the average \eqref{eq_example first reduction}, whose complexity is $(2,1,2)$. Here, $R_1$ is the only nonidentity transformation appearing in the $a_1$-position, and it appears twice, so $c_3'=2$ is attained by $R_1$. The next step is therefore to replace one occurrence of $R_1$ by $I$, reducing its multiplicity to one and yielding an average of complexity $(2,1,1)$. This is precisely the complexity of \eqref{eq_example second reduction}.
Continuing in the same manner, after four reduction steps we arrive at the average \eqref{eq_example final step}, whose complexity is $(1,1,1)$. Its convergence now follows directly from Theorem~\ref{thm_PMET}.

\subsection{The complexity reduction procedure}\label{complexity reduction procedure subsection}
In this subsection, we present the {\em complexity reduction procedure}, which is an algorithmic procedure through which proving convergence for a nice collection $\CS$ is reduced to proving convergence for a lower complexity nice collection of actions. The idea of the proof is that we will apply this procedure repeatedly until we are reduced to a collection for which the convergence is known.

Fix a nice collection of actions $\CS$ with complexity $\mathbf{c}>\mathbf{1}$. The procedure consists of multiple {\em complexity reduction steps}.
Every time we apply a reduction step, we reduce $\CS$ to a new nice collection. We denote the new collection, at the $k$-th step, by $\CR_k(\CS)$, where $\CR_0(\CS)=\CS$. We may stop the procedure after any finite number $K$ of steps. We refer to the iteration of these $K$ reduction steps as a {\em $K$-step complexity reduction procedure}. 
We write $\mathbf{c}\longrightarrow\wt{\mathbf{c}}$ to denote that a collection of complexity $\mathbf{c}$ can be reduced to a collection of strictly smaller complexity $\wt{\mathbf{c}}$ after a reduction procedure.

It is not a priori clear that the procedure produces a decrease in complexity after finitely many steps; establishing this is our primary objective.

Each step of the complexity reduction procedure is built by an iterative construction. All the steps are almost identical, with a very small but important difference between the first step and the next ones. We will explicitly describe the first step and then highlight what changes in the subsequent ones.

\textbf{First reduction step.} In what follows, upper indices will be used to denote the iteration step, which we call the {\em level} of the reduction step, except for the parameters $j_d,r_d$, since this would lead to unwieldy notation. We initially fix the letter $D$ to count the number of iterations in the following algorithm, and we explain when it terminates. We also fix a total order $\prec$ on the set of all transformations $\{T_{i,j}\colon i\in[\ell],j\in[m]\}$. 
First we define the {\em data} of the reduction step:
\begin{itemize}    
    \item \underline{Active positions, localizations and pivot}: Let
    $$j_1:=\min\{j\in[m]\colon c(a_j)=c_2\}.$$
    We define an equivalence relation $\sim_1$ on $\CS$ by $S_i\sim_1 S_{i'}$ if $T_{i,j_1}=T_{i',j_1}$, and we refer to the equivalence classes as {\em classes at level $1$}. We define the {\em localization at level $1$}, denoted by $\Loc_1(\CS)$, to be the nonidentity class of minimum cardinality whose representative is $\prec$-smallest among all such classes. We denote this representative by $T_{i_0,j_1}$. Note that this is the $\prec$-smallest transformation that attains $c_3$. If $|\Loc_1(\CS)|=1$, in which case $\Loc_1(\CS)=\{S_{i_0}\}$, then we let $D=1$, i.e., the reduction step terminates at level $1$, and we define the {\em pivot} of the reduction step to be $S_{i_0}$. Otherwise, we move on to the next level.

    Let $d>1$ with $|\Loc_{d-1}(\CS)|>1$, and suppose that $j_{d-1}$, $\Loc_{d-1}(\CS)$ and $T_{i_0,j_{d-1}}$ are defined. Let 
    $$j_d := \min\{j\in[m]\setminus\{j_1,\ldots,j_{d-1}\}\colon c(a_{j};\Loc_{d-1}(\CS))=c_2(\Loc_{d-1}(\CS))\}.$$
    Thus $j_d$ is the first remaining nonconstant position which realizes the second complexity parameter of $\Loc_{d-1}(\CS)$.
    We define an equivalence relation $\sim_d$ on $\Loc_{d-1}(\CS)$ by $S_i\sim_d S_{i'}$ if $T_{i,j_d}=T_{i',j_d}$, and we refer to the equivalence classes as {\em classes at level $d$}. 
    We define the {\em localization at level $d$}, denoted it by $\Loc_d(\CS)$, to be the nonidentity class of minimum cardinality whose representative is $\prec$-smallest among all such classes. We denote this representative by $T_{i_0,j_d}$. 
    Note that this is the $\prec$-smallest transformation that attains $c_3(\Loc_{d-1}(\CS))$. 
    If $|\Loc_d(\CS)|=1$, in which case $\Loc_d(\CS)=\{S_{i_0}\}$, then we let $D=d$, i.e., the reduction step terminates at level $d$, and we define the {\em pivot} of the reduction step to be $S_{i_0}$. 
    Otherwise, we move on to the next level.

    We remark that the reduction step always terminates. To see this, if $|\Loc_d(\CS)|=1$ for some $d<m$, then we take $D$ the first such $d$. Otherwise, $|\Loc_m(\CS)|=1$, thus, $D=m$.

    \item \underline{Other data}: 
    Let 
    $$r_1:=c(a_{j_1})-1=c_2-1,$$
    and consider the set
    $$\{T_{i,j_1}\neq I\colon i\in[\ell]\} := \{T_{i_0,j_1},T_{i_1^{(1)},j_1},\ldots,T_{i_{r_1}^{(1)},j_1}\}.$$
    We partition $\{i\in[\ell]\colon T_{i,j_1}\neq I\}$ into the pairwise disjoint sets
    $$C_t^{(1)}:=\{i\in [\ell]\colon T_{i,j_1}=T_{i_t^{(1)},j_1}\}, \quad\forall 0\leq t\leq r_1,$$
    so that $C_t^{(1)}$ identifies the actions of the $T_{i_t^{(1)},j_1}$-class. We remark that 
    $$\Loc_1(\CS) = T_{i_0,j_1}\text{-class} = \{S_i\colon i\in C_0^{(1)}\}.$$
    
    Let $d>1$ with $|\Loc_{d-1}(\CS)|>1$, and suppose that $r_{d-1}$, $(T_{i_t^{(d-1)},j_{d-1}})_{0\leq t\leq r_{d-1}}$ and $(C_t^{(d-1)})_{0\leq t\leq r_{d-1}}$ are defined. Let 
    $$r_d := c(a_{j_d};\Loc_{d-1}(\CS)) + \1_{\{\exists\, i \in C_0^{(d-1)}\colon T_{i,j_d}=I\}} - 1,$$
    and consider the set 
    $$\{T_{i,j_d}\colon i\in C_0^{(d-1)}\} := \{T_{i_0,j_d},T_{i_1^{(d)},j_d},\ldots,T_{i_{r_d}^{(d)},j_d}\}.$$
    We partition $C_0^{(d-1)}$ into the pairwise disjoint sets 
    $$C_t^{(d)}:=\{i\in C_0^{(d-1)}\colon T_{i,j_d}=T_{i_t^{(d)},j_d}\}, \quad\forall 0\leq t\leq r_d,$$
    so that $C_t^{(d)}$ identifies the actions of the $T_{i_t^{(d)},j_d}$-class of $\Loc_{d-1}(\CS)$. We remark that
    $$\Loc_d(\CS) = T_{i_0,j_d}\text{-class of}\ \Loc_{d-1}(\CS) = \{S_i\colon i\in C_0^{(d)}\}.$$
\end{itemize}

\begin{Remark}\label{remark about complexity vectors}
    After relabeling the positions of the current reduction step we may assume that $j_d=d$ for all $d\in[D]$. With this convention, and setting $\Loc_0(\CS)=\CS$, we have
    $$c(a_d;\Loc_{d-1}(\CS))=c(a_{j_d};\Loc_{d-1}(\CS))=c_2(\Loc_{d-1}(\CS)),$$
    for all $d\in[D]$. This applies only to the first reduction step. After this, we are no longer allowed to relabel the positions to avoid confusion concerning the reduction step data.
\end{Remark}

Every complexity reduction step comprises two phases: a seminorm control in specific directions, and then reducing our problem to the convergence of a new collection via magic extensions.

First, we obtain the seminorm control. This is done through several applications of our orthogonality criterion as follows. Let $\CF=\{F_i\in L^\infty(\mu)\colon |F_i|=1, i\in[\ell]\}$. We also define $C_0^{(0)} = \{i\in[\ell]\colon T_{i,1}\neq I\}$.

\textbf{$a_1$-position:}
We apply Lemma~\ref{Katai for actions} for $\CP(a_1)$. Then $\limsup_{N\to\infty}\|A_N(\CS,\CF)\|_2^2$ is bounded by
$$
    \E_{k_{0,0},k_{0,1}\in[K]}\limsup_{N\to\infty}\Re\Big(\E_{n\in[N]}\int_X \prod_{i\in C_0^{(0)}}\prod_{j=1}^m T_{i,j}^{a_j(n)}\big(T_{i,1}^{k_{0,0}} F_i\cdot T_{i,1}^{k_{0,1}}\overline{F_i}\big)\, d\mu\Big).
$$
We have eliminated all actions that have the identity transformation in the $a_1$-position.
Now, we compose with 
$S_{i_1^{(1)},n}^{-1} = \prod_{j=1}^m T_{i_1^{(1)},j}^{-a_j(n)}$ (note that $i_1^{(1)}\in C_0^{(0)}$) and we use the Cauchy-Schwarz inequality to bound the above by
$$\E_{k_{0,0},k_{0,1}\in[K]}\limsup_{N\to\infty}\Bigg\|\E_{n\in[N]}\prod_{\substack{i\in C_0^{(0)} \\ i\neq i_1^{(1)}}}\prod_{j=1}^m \big(T_{i,j}T_{i_1^{(1)},j}^{-1}\big)^{a_j(n)}\big(T_{i,1}^{k_{0,0}} F_i\cdot T_{i,1}^{k_{0,1}}\overline{F_i}\big)\Bigg\|_2.$$
Note that for the $i$-th actions with $i\in C_1^{(1)}$, we have by the definition of this set that $T_{i,1}T_{i_1^{(1)},1}^{-1} = I$.
We square this expression and then apply Lemma~\ref{Katai for actions} for $\CP(a_1)$. Then all the $i$-th actions with $i\in C_1^{(1)}$ vanish and the $\limsup_{N\to\infty}\norm{A_{N}(\CS,\CF)}_2^4$ is bounded by
\begin{multline*}
    \E_{\substack{k_{t,u_t}\in[K] \\ t\in \{0,1\},\\\underline{u}\in\{0,1\}^2}}\limsup_{N\to\infty}\Re\Bigg(\E_{n\in[N]}\int_X \prod_{i\in C_0^{(0)}\setminus C_1^{(1)}} \prod_{j=1}^m T_{i,j}^{a_j(n)} \\
    \Bigg(\prod_{\underline{u}\in\{0,1\}^2}T_{i,1}^{k_{0,u_0}}\big(T_{i,1}T_{i_1^{(1)},1}^{-1}\big)^{k_{1,u_1}} \CC^{|\underline{u}|}F_i\Bigg)\,d\mu\Bigg),
\end{multline*}
where we have with $S_{i_1^{(1)},n}$. 

Now, we compose with $S_{i_2^{(1)},n}^{-1} = \prod_{j=1}^m T_{i_2^{(1)},j}^{-a_j(n)}$ (note that $i_2^{(1)}\in C_0^{(0)}\setminus C_1^{(1)}$) and we use the Cauchy-Schwarz inequality to get a bound by
\begin{multline*}
    \E_{\substack{k_{t,u_t}\in[K] \\ t\in \{0,1\},\\\underline{u}\in\{0,1\}^2}}\limsup_{N\to\infty}\Bigg\|\E_{n\in[N]} \prod_{\substack{i\in C_0^{(0)}\setminus C_1^{(1)} \\ i\neq i_2^{(1)}}}\prod_{j=1}^m\big(T_{i,j}T_{i_2^{(1)},j}^{-1}\big)^{a_j(n)} \\
    \Bigg(\prod_{\underline{u}\in\{0,1\}^2} T_{i,1}^{k_{0,u_0}}\big(T_{i,1}T_{i_1^{(1)},1}^{-1}\big)^{k_{1,u_1}} \CC^{|\underline{u}|}F_i\Bigg)\Bigg\|_2.
\end{multline*}

By the definition of the set $C_2^{(1)}$, for all $i\in C_2^{(1)}$, we have $T_{i,1}T_{i_2^{(1)},1}^{-1}=I$.
We square this expression and then apply Lemma~\ref{Katai for actions} for $\CP(a_1)$. Then all the $i$-th actions with $i\in C_1^{(1)}\cup C_2^{(1)}$ vanish and the $\limsup_{N\to\infty}\|A_N(\CS,\CF)\|_2^8$ is bounded by
\begin{multline*}
    \E_{\substack{k_{t,u_t}\in[K] \\ t\in\{0,1,2\},\\\underline{u}\in\{0,1\}^3}}\limsup_{N\to\infty}\Re\Bigg(\E_{n\in[N]}\int_X \prod_{i\in C_0^{(0)}\setminus (C_1^{(1)}\cup C_2^{(1)})}\prod_{j=1}^m T_{i,j}^{a_j(n)} \\
    \Bigg(\prod_{\underline{u}\in\{0,1\}^3} T_{i,1}^{k_{0,u_0}}\big(T_{i,1}T_{i_1^{(1)},1}^{-1}\big)^{k_{1,u_1}}\big(T_{i,1}T_{i_2^{(1)},1}^{-1}\big)^{k_{2,u_2}} \CC^{|\underline{u}|}F_i\Bigg)\,d\mu\Bigg),
\end{multline*}
where we composed with $S_{i_2^{(1)},n}$ in the last integral.

We continue this process until all the $i$-th actions with $i\in\bigcup_{t=1}^{r_1} C_t^{(1)} = C_0^{(0)}\setminus C_0^{(1)}$ have vanished, thus, the only actions remaining have $T_{i,1}=T_{i_0,1}$. By definition, this will happen after $c(a_1)=r_1+1$ applications of Lemma~\ref{Katai for actions}. Then, we have bounded the $2^{r_1+1}$ power of our initial average by 
\begin{align*}
    & \E_{\substack{k_{t,u_t}\in[K] \\ t\in \{0,\ldots, r_1\},\\\underline{u}\in\{0,1\}^{r_1+1}}}\limsup_{N\to\infty}\Re\Bigg(\E_{n\in[N]}\int_X \prod_{i\in C_0^{(1)}} \left(T_{i,1}^{a_1(n)}\prod_{j=2}^m T_{i,j}^{a_j(n)}\right) \\
    & \hspace{6cm} \Bigg(\prod_{\underline{u}\in\{0,1\}^{r_1+1}}\bigg(T_{i,1}^{k_{0,u_0}}\prod_{t=1}^{r_1} \big(T_{i,1}T_{i_t^{(1)},1}^{-1}\big)^{k_{t,u_t}}\bigg)\CC^{|\underline{u}|}F_i\Bigg)\,d\mu \Bigg)\\
    & = \E_{\substack{k_{t,u_t}\in[K] \\ t\in\{0,\ldots, r_1\},\\\underline{u}\in\{0,1\}^{r_1+1}}}\limsup_{N\to\infty}\Re\Bigg(\E_{n\in[N]}\int_X \prod_{i\in C_0^{(1)}} \prod_{j=2}^m T_{i,j}^{a_j(n)} \\
    & \hspace{6cm} \Bigg(\prod_{\underline{u}\in\{0,1\}^{r_1+1}}\bigg(T_{i,1}^{k_{0,u_0}}\prod_{t=1}^{r_1} \big(T_{i,1}T_{i_t^{(1)},1}^{-1}\big)^{k_{t,u_t}}\bigg)\CC^{|\underline{u}|}F_i\Bigg)\,d\mu\Bigg).
\end{align*}
Since $\Loc_1(\CS) = \{S_i\colon i\in C_0^{(1)}\}$, the actions surviving in the integral are exactly the ones in the $1$-localization.
This concludes the first step, as we have completely eliminated the $a_1$-position from our actions. If $D=1$, the procedure ends here. Otherwise, we continue to the next position.

\textbf{$a_2$-position:} We deal with the $a_2$-position in a similar fashion. We start the same process as before, this time applying Lemma~\ref{Katai for actions} for the set $\CP(a_2)$, until we eliminate the $a_2$-position completely. First, we compose with $\prod_{j=2}^m T_{i_1^{(2)},j}^{-a_j(n)}$ (note that $i_1^{(2)}\in C_0^{(1)}$), and then we use the Cauchy-Schwarz inequality to bound the last expression by
\begin{multline*}
    \E_{\substack{k_{t,u_t}\in[K] \\ t\in\{0,\ldots, r_1\},\\
    \underline{u}\in\{0,1\}^{r_1+1}}}\limsup_{N\to\infty}\Bigg\|\E_{n\in[N]} \prod_{\substack{i\in C_0^{(1)} \\ i\neq i_1^{(2)}}}\prod_{j=2}^m \big(T_{i,j}T_{i_1^{(2)},j}^{-1}\big)^{a_j(n)} \\
    \Bigg(\prod_{\underline{u}\in\{0,1\}^{r_1+1}}\bigg(T_{i,1}^{k_{0,u_0}}\prod_{t=1}^{r_1} \big(T_{i,1}T_{i_t^{(1)},1}^{-1}\big)^{k_{t,u_t}}\bigg)\CC^{|\underline{u}|}F_i\Bigg)\Bigg\|_2.
\end{multline*}
We square this expression and then apply Lemma~\ref{Katai for actions} for $\CP(a_2)$. Then all the $i$-th actions with $i\in C_1^{(2)}$ vanish and the $\limsup_{N\to\infty}\|A_N(\CS,\CF)\|_2^{2^{r_1+2}}$ is bounded by
\begin{multline*}
    \E_{\substack{k_{t,u_t}\in[K] \\ t\in \{0,\ldots, r_1+1\},\\\underline{u}\in\{0,1\}^{r_1+2}}}\limsup_{N\to\infty}   \Re\Bigg(\E_{n\in[N]}\int_X \prod_{i\in C_0^{(1)}\setminus C_1^{(2)}} \prod_{j=2}^m T_{i,j}^{a_j(n)} \\
    \Bigg(\prod_{\underline{u}\in\{0,1\}^{r_1+2}}\bigg(T_{i,1}^{k_{0,u_0}}\prod_{t=1}^{r_1} \big(T_{i,1}T_{i_t^{(1)},1}^{-1}\big)^{k_{t,u_t}} \big(T_{i,2}T_{i_1^{(2)},2}^{-1}\big)^{k_{r_1+1,u_{r_1+1}}}\bigg)\CC^{|\underline{u}|}F_i\Bigg)\,d\mu\Bigg),
\end{multline*}
where we have composed with $\prod_{j=2}^m T_{i_1^{(2)},j}^{a_j(n)}$.
Then, we compose with $\prod_{j=2}^m T_{i_2^{(2)},j}^{-a_j(n)}$, and then we use the Cauchy-Schwarz inequality, and then Lemma~\ref{Katai for actions} to get an analogous bound for $\limsup_{N\to\infty}\|A_N(\CS,\CF)\|_2^{2^{r_1+3}}$, where all the $i$-th actions with $i\in C_1^{(2)}\cup C_2^{(2)}$ have vanished. We continue until all the $i$-th actions with $i\in\bigcup_{t=1}^{r_2} C_t^{(2)} = C_0^{(1)}\setminus C_0^{(2)}$ have been eliminated. Ultimately,  we have bounded the $\limsup_{N\to\infty}\|A_N(\CS,\CF)\|_2^{2^{r_1+r_2+1}}$ power of our initial average by  
\begin{multline*}
    \E_{\substack{k_{t,u_t}\in[K] \\ t\in\{0,\ldots, r_1+r_2\},\\\underline{u}\in\{0,1\}^{r_1+r_2+1}}}\limsup_{N\to\infty}\Re\Bigg(\E_{n\in[N]}\int_X \prod_{i\in C_0^{(2)}} \prod_{j=3}^m T_{i,j}^{a_j(n)} \\
    \Bigg(\prod_{\underline{u}\in\{0,1\}^{r_1+r_2+1}}\bigg(T_{i,1}^{k_{0,u_0}}\prod_{t=1}^{r_1} \big(T_{i,1}T_{i_t^{(1)},1}^{-1}\big)^{k_{t,u_t}}\prod_{t=r_1+1}^{r_1+r_2}\big(T_{i,2}T_{i_{t-r_1}^{(2)},2}^{-1}\big)^{k_{t,u_t}}\bigg)\CC^{|\underline{u}|}F_i\Bigg)\,d\mu\Bigg).
\end{multline*}

If $D=2$, the procedure ends here. Otherwise, we continue in the next position, in which we repeat the process exactly as we did for the previous two positions.

\textbf{Final bound:}
We continue through the subsequent active positions until reaching the $D$-th position, where $D$ is the smallest $j$ such that $C_0^{(j)}$ is a singleton. Writing $C_0^{(D)}=\{i_0\}$, only the action indexed by $i_0$ remains, and its remaining positions disappear under the final integral by measure preservation. We set $r=\sum_{j=1}^D r_j$, and then the $\limsup_{N\to\infty}\|A_N(\CS,\CF)\|_2^{2^{r+1}}$ is bounded by 
\begin{multline*}
    \E_{\substack{k_{t,u_t}\in[K] \\ t\in\{0,\ldots, r\},\\\underline{u}\in\{0,1\}^{r+1}}}\limsup_{N\to\infty}\Re\Bigg(\E_{n\in[N]}\int_X \prod_{i\in C_0^{(D)}}\prod_{j=D+1}^m T_{i,j}^{a_j(n)} \\
    \Bigg(\prod_{\underline{u}\in\{0,1\}^{r+1}}\bigg(T_{i,1}^{k_{0,u_0}}\prod_{d=1}^D\prod_{t=r_1+\ldots+r_{d-1}+1}^{r_1+\ldots+r_d}\Big(T_{i,d}T_{i_{t- (r_1+\dots+r_{d-1})}^{(d)},d}^{-1}\ \Big)^{k_{t,u_t}}\bigg)\CC^{|\underline{u}|}F_i\Bigg)\,d\mu\Bigg),
\end{multline*}
where we have set $r_0=0$. Using that $C_0^{(D)} = \{i_0\}$ and composing with $\prod_{j=D+1}^m T_{i_0,j}^{-a_j(n)}$ in the integral, we rewrite this as
\begin{multline*}
    \E_{\substack{k_{t,u_t}\in[K] \\ t\in \{0,\ldots, r\},\\\underline{u}\in\{0,1\}^{r+1}}}\int_X
    \prod_{\underline{u}\in\{0,1\}^{r+1}}\Bigg(T_{i_0,1}^{k_{0,u_0}}\prod_{d=1}^D\prod_{t=r_1+\ldots+r_{d-1}+1}^{r_1+\ldots+r_d}\Big(T_{i_0,d}T_{i_{t- (r_1+\ldots+r_{d-1})}^{(d)},d}^{-1}\Big)^{k_{t,u_t}}\Bigg)\\
    \CC^{|\underline{u}|}F_{i_0}\,d\mu.
\end{multline*}
By the definition of the box seminorms, the limit of this last expression as $K\to\infty$ is equal to the $2^{r+1}$-th power of the box seminorm of $F_{i_0}$ with respect to the transformations appearing in the products above (we have $r+1$ directions and $2^{r+1}$ cube factors in total).
More precisely, we have shown that for any collection $\CF$ of functions with values in the unit circle, we have
\begin{equation}\label{seminorm_bound_eq}
    \limsup_{N\to\infty}\|A_N(\CS,\CF)\|_2 \ll\nnorm{F_{i_0}}_{T_{i_0,1},T_{i_0,1}T_{i_1^{(1)},1}^{-1},\ldots,T_{i_0,1}T_{i_{r_1}^{(1)},1}^{-1},\ldots,T_{i_0,D}T_{i_1^{(D)},D}^{-1},\ldots,T_{i_0,D}T_{i_{r_D}^{(D)},D}^{-1}}.    
\end{equation}
By the same argument, a similar bound holds for all collections of 1-bounded functions. 

\textbf{Reduction:}
We now exploit the seminorm control to reduce our problem to the convergence of a new average, for which we will eventually prove that it has a lower complexity than the original. This guarantees that all the transformations in the above seminorm are pairwise distinct.

Let $\{T_i\colon i\in[k]\} = \{T_{i,j}\colon i\in[\ell],j\in[m]\}$, for some $k\in[\ell m]$ be an enumeration of all the transformations appearing in the initial average.
In order to extract some valuable information from the previous seminorm, we will use the machinery of magic extensions. Consider the $\Z^k$-system $\textbf{X} = (X,\mu,(T_i)_{i\in[k]})$.

The directions appearing in the preceding box seminorm are
$$
T_{i_0,1},\qquad
T_{i_0,d} T_{i_t^{(d)},d}^{-1}
\quad(d\in[D],\ t\in[r_d]).
$$
In order to apply Proposition \ref{prop_Nikos_and_Borys}, we need to show that this collection extends to a $\Z$-basis for the entire collection $\{T_{i,j}\colon i\in[\ell],j\in[m]\}$.
By condition \ref{condition4}, the transformations $T_{i,j}$ freely generate a free abelian group.

We do this as follows. In position $a_1$, we replace the transformations $T_{i_0,1}$ and $T_{i_t^{(1)},1}$, for $t\in[r_1]$,
by $T_{i_0,1}$ and $T_{i_0,1}T_{i_t^{(1)},1}^{-1}$ respectively.
For each position $2\leq d\leq D$ we replace the nonidentity transformations $T_{i_t^{(d)},d}$, for $t\in[r_d]$, by $T_{i_0,d}T_{i_t^{(d)},d}^{-1}$.
We keep all other nonidentity elements of the set $\{T_{i,j}\colon i\in[\ell],j\in[m]\}$ unchanged. This gives a new basis for the free abelian group $\{T_{i,j}\colon i\in[\ell],j\in[m]\}$, which we denote as
$\{P_1,\ldots, P_q\}$, where $q$ is at most $ \ell m$ since we have removed identity transformations. 

For simplicity, we denote the above box seminorm by $\nnorm{F_{i_0}}$.
After reordering the $P_i$, the directions of the box seminorm are
$P_1,\ldots,P_{r+1}$. Then, by Proposition~\ref{prop_Nikos_and_Borys}, $\textbf{X}$ admits an extension $\textbf{X}^\ast=(X^\ast,\mu^\ast,(P_i^\ast)_{i\in[q]})$ that is magic with respect to $T_{i_0,1}^\ast$, $T_{i_0,d}^\ast (T_{i_t^{(d)},d}^\ast)^{-1}$, with $d\in[D]$, $t\in[r_d]$. It is straightforward that all the nice conditions \ref{condition1}--\ref{condition4} are lifted in the extension, as any relations between the maps on the extension would produce a relation between the corresponding transformations on the original system via the natural factor map. Note that we can lift the maps $T_{i,j}$ to the extension: if $$T_{i,j}=P_1^{n_1}\ldots P_q^{n_q}, \qquad  n_1,\ldots n_q\in \Z,$$
then we consider the lift 
$$T^\ast_{i,j}= (P^\ast_1)^{n_1}\ldots (P^\ast_q)^{n_q}.$$
As the representation of $T_{i,j}$ in the basis comprised of the $P_i$ is unique, this gives a unique lift $T^\ast_{i,j}$ in the extension.  

We use the notation $\nnorm{\cdot}^\ast$ to denote the box seminorm in the system $\textbf{X}^\ast$ with the same directions as the ones appearing in $\nnorm{F_{i_0}}$.
Then, by Definition~\ref{defn magic}, the $\sigma$-algebra 
$$\CZ^\ast := \I(T_{i_0,1}^\ast)\vee\bigg(\bigvee_{d=1}^D\bigvee_{t=1}^{r_d}\I({T_{i_0,d}^\ast (T_{i_t^{(d)},d}^\ast)^{-1}})\bigg)$$
satisfies the following:
for any $F^\ast\in L^\infty(\mu^\ast)$, if $\mathcal{E}_{\mu^\ast}(F^\ast|\CZ^\ast)=0$, then $\nnorm{F^\ast}^\ast=0$. 
By its definition, $L^2(\CZ^\ast,\mu^\ast)$ is the closed linear span of products $G_1^\ast\cdot\ldots \cdot G_{r+1}^\ast$, where each $G_i^\ast\in L^\infty(\mu^\ast)$ is invariant under the corresponding transformation appearing in the seminorm $\nnorm{F_{i_0}^\ast}^\ast$, and $r+1$ is the number of these transformations.

Let $\CS^\ast=\{S^\ast_1,\ldots,S^\ast_\ell\}$ be the nice collection of actions, where each $S_i^\ast$ is defined by replacing every generator $T$ from $\textbf{X}$ with the corresponding transformation $T^\ast$ from $\textbf{X}^\ast$.
It is straightforward that $(X^\ast,\mu^\ast,S^\ast_1,\ldots,S^\ast_\ell)$ is an extension of $(X,\mu,S_1,\ldots,S_\ell)$. 
Thus, it suffices to establish the convergence of the average $A_N(\CS,\CF)$ after passing to this extension. Hence, we shall prove that for any collection $\CF^\ast$ of functions $F^\ast_1,\ldots,F^\ast_\ell\in L^\infty(\mu^\ast)$ with $|F^\ast_i|=1$ for all $i\in[\ell]$, the average
\begin{equation}\label{average on extension}
    A_N(\CS^\ast,\CF^\ast) = \E_{n\in[N]}\prod_{i=1}^\ell\Bigg(\prod_{j=1}^m {T^\ast_{i,j}}^{a_j(n)}\Bigg)F^\ast_i
\end{equation}
converges in $L^2(\mu^\ast)$.
Since we have proven 
\begin{equation*}
    \limsup_{N\to\infty}\|A_N(\CS^\ast,\CF^\ast)\|_2 \ll {\nnorm{F^\ast_{i_0}}^\ast},
\end{equation*}
then by decomposing $$F^{\ast}_{i_0}=\mathcal{E}_{\mu^\ast}(F^\ast_{i_0}|\CZ^\ast )+(F^{\ast}_{i_0}-\mathcal{E}_{\mu^\ast}(F^\ast_{i_0}|\CZ^\ast))$$
and using the seminorm bound (note that the second function has zero box seminorm), we reduce to the case where  $F^\ast_{i_0}$ is measurable with respect to $\CZ^{\ast}$.

By the definition of this $\sigma$-algebra, the linear span of products of the form 
$$F^\ast_0\cdot \prod_{d=1}^D\prod_{t=1}^{r_d}F^\ast_{t,d},$$
where all factors belong to $L^\infty(\mu^\ast)$, $F^\ast_0$ is $T^\ast_{i_0,1}$-invariant, and each $F^\ast_{t,d}$ is $T^\ast_{i_0,d}(T^\ast_{i_t^{(d)},d})^{-1}$-invariant, is dense in $L^2(\mu^\ast,\CZ^\ast)$. With the other functions fixed and bounded, the averages depend continuously on $F_{i_0}^\ast$ in the $L^2$-norm, uniformly in $N$. Thus, by density and linearity, it suffices to consider the case where $F_{i_0}^\ast$ is a single product of the above form. Substituting this function into \eqref{average on extension} we obtain 
\begin{multline*}
    \E_{n\in[N]}\prod_{\substack{i=1 \\ i\neq i_0}}^\ell\Bigg(\prod_{j=1}^m {T^\ast_{i,j}}^{a_j(n)}\Bigg)F^\ast_i \cdot \Bigg(\prod_{j=2}^m {T^\ast_{i_0,j}}^{a_j(n)}\Bigg) F^\ast_0 \\
    \cdot \prod_{d=1}^D\prod_{t=1}^{r_d}\Bigg({T^\ast_{i_t^{(d)},d}}^{a_d(n)}\prod_{\substack{j=1 \\ j\neq d}}^m {T^\ast_{i_0,j}}^{a_j(n)}\Bigg)F^\ast_{t,d}.
\end{multline*}
Thus, it suffices to prove that the latter average converges in $L^2(\mu^\ast)$.

From now on, for notational convenience and since we will repeatedly pass to extensions, we will omit the asterisk notation. That is, we will be working on $(X,\mu,S_1,\ldots,S_\ell)$, even though we are actually working with extensions of the form $(X^\ast,\mu^\ast,S^\ast_1,\ldots,S^\ast_\ell)$.

\begin{Remark}\label{rem_nice_and_reduction_step}
    The first phase of the reduction step, that is, obtaining seminorm control, does not rely at all on the assumptions that $\CS$ is nice. 
    Only in the second phase, we use condition \ref{condition4} to ensure that the transformations in the seminorm generate a free abelian group, thus, allowing us to employ Proposition~\ref{prop_Nikos_and_Borys}.
    Conditions \ref{condition3} and \ref{condition4} ensure that the new collection produced by the reduction step is nice, allowing the iteration of the procedure to proceed.  
    The other two conditions are just for convenience.
    For instance, the reduction procedure works even if \ref{condition1} is not satisfied. In this case, there are some positions, say $r$, in which all actions have the same transformation in each position. 
    By reordering, we may assume these positions are from $a_1$ to $a_r$, and so we start the reduction step from the $a_{r+1}$-position, proceeding to $a_{r+2}$, and so on. 
    In particular, if $\CS$ is such a collection, let $\CS'$ be the collection consisting of actions from $\CS$ without the first $r$ positions. 
    Then, the complexity reduction for $\CS$ and $\CS'$ are essentially the same. 
    More precisely, for any $k\in\N$, we define $\CR_k(\CS)$ to consist of the same action as $\CR_k(\CS')$ with the first $r$ positions restored. 
\end{Remark}

\textbf{Higher steps.} Suppose we have applied $k-1$ reduction steps on $\CS$, for $k\geq2$, and we have reduced it to $\wt{\CS} = \{\wt{S}_i\colon i\in [\wt{\ell}]\} = \CR_{k-1}(\CS)$.
For the reduction step taking $\CR_{k-2}(\CS)$ to $\wt{\CS}$, let $D$ be the level at which the step terminated, $(j_d)_{d\in[D]}$ be the active positions, $(\Loc_d(\CR_{k-2}(\CS)))_{d\in[D]}$ be the localizations and $P=(T_1,\ldots,T_D,\ldots,T_m)$ be the pivot.
We decorate with a tilde all the data that are relative to the collection $\wt{\CS}$. 
The reduction step proceeds analogously to the first, with $\CS$ replaced by $\wt{\CS}$, subject to one crucial difference explicitly described below.

It is clear that $P$ has now vanished from $\wt{\CS}$, but its class at some level might have survived. Writing $P=(P_1,\ldots,P_m)$, we define 
$$
    A_0 = \wt{S}, 
    \quad
    A_d = \{\wt{S}_i\in\wt{S}\colon \wt{T}_{i,j_s}=P_{j_s}\ \forall s\in[d]\},
$$
and
$$
    d_{\max} = \max\{0\leq d\leq D\colon A_d\neq\emptyset\}.
$$

In fact, $d_{\max}=D-1$ and $d_{\max}=0$ if and only if the transformation $T_{j_1}$ has vanished from position $j_1$ in $\wt{\CS}$. Since $j_1$ realizes the second complexity parameter of $\CR_{k-2}(\CS)$, this implies that
$c_2(\wt{\CS})=c_2(\mathcal R_{k-2}(\CS))-1$
if $c_2(\CR_{k-2}(\CS))>1$. Otherwise, position $j_1$ becomes constant and is removed, so 
$c_1(\wt{\CS})=c_1(\CR_{k-2}(\CS))-1$.

For all $d\leq d_{\max}$, we choose $\wt{j}_d=j_d$ and $\Loc_d(\wt{\CS})$ to be the $T_{j_d}$-class within $\Loc_{d-1}(\wt{\CS})$. In other words, the localization at level $d$ consists of actions that in the $d$-th active position, which is the same as in the preceding step, have the same transformation as in the corresponding localization of the preceding step. Note that $\Loc_d(\wt{\CS})$ need not be a nonidentity class of minimum cardinality within $\Loc_{d-1}(\wt{\CS})$; equivalently, $T_{j_d}$ need not attain $c_3(\Loc_{d-1}(\wt{\CS}))$.

For $d>d_{\max}$, $\wt{j}_d$ is chosen by the selection rule of the first step. Thus, the new pivot is uniquely determined (recall that ties are broken by the total order $\prec$ described in the first step) so that it satisfies $\wt{P}\sim_{d_{\max}} P$, $\wt{P}\not\sim_{d_{\max}+1}P$. To put it simply, the new pivot agrees with the previous one at the first $d_{\max}$ active positions.

When we initiate the procedure for $\CS$, suppose that the pivot of the first reduction step belongs to the $T$-class at level $1$, for some transformation $T$. 
As long as the $T$-class survives, the first active position remains fixed. In fact, it is always the $a_1$-position, since in the first reduction step we have agreed to start from that position. 
Consequently, the pivot remains in the $T$-class throughout the reduction procedure until that class vanishes, at which point the data are chosen anew according to the rule used in the first step. 
We remark that the initial pivot class survives, and hence the pivot is chosen from that class and the initial position remains fixed, as long as the second complexity parameter does not drop (or, when the second complexity parameter equals $1$, as long as the first complexity parameter does not drop).

\begin{Proposition}[Complexity reduction step]\label{complexity reduction procedure thm}
Let $\CS=\{S_i\colon i\in[\ell]\}$ be a nice collection of actions on a probability space $\xm$ where $S_i=(T_{i,1},\ldots,T_{i,m})$, for each $i\in[\ell]$, relative to $a_1,\ldots,a_m$, and with complexity vector $\mathbf{c}=(c_1,c_2,c_3)$, with $c_1=m$. 
Let $k\in\N$ and $\CR_{k-1}(\CS) = \{\wt{S}_i\colon i\in[\wt{\ell}]\}$, with all associated transformations similarly decorated with tildes, and assume that this collection is nice. 
Finally, let $D\in\N$ be the terminating level of the $k$-th complexity step for $\CS$, and consider $(j_d)_{d\in[D]}$, $(r_d)_{d\in[D]}$ and $(\wt{T}_{i_t^{(d)},j_d})_{0\leq t\leq r_d,d\in[D]}$ as in the procedure.
Then there exists $i_0\in[\wt{\ell}]$ such that the following hold:
\begin{enumerate}[label=(\Alph*)]
    \item\label{item1} $\wt{T}_{i_0,j_d}\neq I$ for all $d\in[D]$, $\wt{T}_{i_t^{(1)},j_1}\neq I$ for all $t\in[r_1]$, and $\wt{T}_{i_t^{(d)},j_d}\neq \wt{T}_{i_0,j_d}$ for all $d\in[D]$ and all $t\in[r_d]$.
    \item\label{item2} For any $\CF=\{F_i\in L^\infty(\mu)\colon |F_i|=1, i\in[\wt{\ell}]\}$, we have
    \begin{multline}\label{seminorm bound eq}
        \limsup_{N\to\infty}\|A_N(\wt{\CS},\CF)\|_2 \\
        \ll\nnorm{F_{i_0}}_{\wt{T}_{i_0,j_1},\wt{T}_{i_0,j_1}\wt{T}_{i_1^{(1)},j_1}^{-1},\ldots,\wt{T}_{i_0,j_1}\wt{T}_{i_{r_1}^{(1)},j_1}^{-1},\ldots,\wt{T}_{i_0,j_D}\wt{T}_{i_1^{(D)},j_D}^{-1},\ldots,\wt{T}_{i_0,j_D}\wt{T}_{i_{r_D}^{(D)},j_D}^{-1}},
    \end{multline}
    and the transformations in the seminorm freely generate an abelian group.
    \item\label{item3} We have
        \begin{multline}\label{eq_Rk}
            \CR_k(\CS)=\{\wt{S}_i\colon i\in[\wt{\ell}],i\neq i_0\}\cup\{(\wt{S}_{i_0,n}\wt{T}_{i_0,j_1}^{-a_{j_1}(n)})_{n\in\N}\}\\
            \cup\Bigg\{\Bigg(\wt{T}_{i_t^{(d)},j_d}^{a_{j_d}(n)}\prod_{\substack{j=1 \\ j\neq j_d}}^m \wt{T}_{i_0,j}^{a_j(n)}\Bigg)_{n\in\N}\colon d\in[D], t\in[r_d]\Bigg\},
        \end{multline}
        and the collection $\CR_k(\CS)$ is nice.
\end{enumerate}
\end{Proposition}

The proof of the last assertion of \ref{item3} is omitted since it is quite straightforward to verify from \eqref{eq_Rk}.

\begin{Remark}\label{same c_1,c_2 and transformation}
    By construction, in the first step we replace the pivot action which belongs in the $T_{i_0,1}$-class. As long as this class survives, we keep replacing actions from there. Hence, as long as this class survives, both $c_1,c_2$ remain stable, and vice versa. Then the first position remains $j_1=1$. In addition, $c_1$ and $c_2$ cannot drop until $c_3$ drops to $1$ first. When this class vanishes, $c_2$ decreases by $1$, if $c_2>1$, otherwise $c_1$ decreases by $1$. Similarly, we have the analogous situation with the complexity parameters $c_2$ and $c_3$.
\end{Remark}

\subsection{The proof of Theorem~\ref{main_theorem_simple}}
\label{subsection_proof}

To prove Theorem~\ref{main_theorem_simple}, it suffices to establish the next proposition.

\begin{Proposition}\label{arrows prop}
    The following hold for a nice collection of actions with complexity $\mathbf{c}=(c_1,c_2,c_3)\in\mathbf{C}$:
    \begin{enumerate}[label=(\alph*)]
        \item\label{hard arrow} If $c_3>1$, then $\mathbf{c}\longrightarrow(c_1,c_2,1)$.

        \item\label{easy arrow 1} If $c_2>1$ and $c_3=1$, then $\mathbf{c}\longrightarrow(c_1,c_2-1,\wt{c}_3)$.

        \item\label{easy arrow 2} If $c_1>1$, $c_2=1$ and $c_3=1$, then $\mathbf{c}\longrightarrow(c_1-1,\wt{c}_2,\wt{c}_3)$.
    \end{enumerate}
\end{Proposition}

We remark that in \ref{easy arrow 1} and \ref{easy arrow 2}, the reduction procedure requires only one step, contrary to what happens in the case of \ref{hard arrow}. This will be clear from the proof of these statements.

Theorem~\ref{main_theorem_simple} follows from Proposition~\ref{arrows prop} and Theorem~\ref{thm_PMET}, since any complexity vector can be eventually reduced to $\mathbf{1}=(1,1,1)$, because any complexity vector $\mathbf{c}=(c_1,c_2,c_3)$ with $c_1=c_2=1$ must have $c_3=1$, which corresponds to an average of the form 
$\E_{n\in[N]} S_nF$,
and this converges in $L^2(\mu)$ by Theorem~\ref{thm_PMET}.
Thus, it remains to establish Proposition~\ref{arrows prop}.

\begin{proof}[Proof of \ref{easy arrow 1} and \ref{easy arrow 2} of Proposition~\ref{arrows prop}]
The proofs of \ref{easy arrow 1} and \ref{easy arrow 2} are essentially identical and each  of them requires just one step in the complexity reduction procedure. We therefore present only the proof of \ref{easy arrow 1}. 

Let $\CS=\{S_i\colon i\in[\ell]\}$ be a nice collection of actions, with $S_i = (T_{i,1},\ldots,T_{i,m})$, for each $i\in[\ell]$, relative to some $a_1,\ldots,a_m$.
Suppose that $\CS$ has complexity vector $\mathbf{c}=(c_1,c_2,c_3)\in\mathbf{C}$ with $c_1=m$, $c_2>1$ and $c_3=1$. We apply the first reduction step. 
By Proposition~\ref{complexity reduction procedure thm} \ref{item1}, there is $i_0\in[\ell]$ such that $c_3$ is attained by $T_{i_0,1}$, which means that $T_{i_0,1}$ appears only once in the $a_1$-position in the actions of $\CS$, since $c_3=1$. 
By Proposition~\ref{complexity reduction procedure thm} \ref{item3}, $\CS$ is reduced to $\CR_1(\CS)$ given by \eqref{eq_Rk}. 
None of the actions in this collection has $T_{i_0,1}$ in the $a_1$-position. 
The set of nonidentity transformations appearing in that position is
$$
    \{T_{i,1}\colon i\in[\ell],\ i\neq i_0,\ T_{i,1}\neq I\}
    =
    \{T_{i_t^{(1)},1}\colon t\in[r_1]\}.
$$
The identity transformation also appears in this position, through $(I,T_{i_0,2},\ldots,T_{i_0,m})$, but does not contribute
to the count of nonidentity transformations.
Since $r_1=c(a_1)-1=c_2-1$, and all other positions retain their
original sets of transformations, we obtain
$c_2(\CR_1(\CS))=c_2-1$, which concludes the proof.
\end{proof}

It remains to prove \ref{hard arrow} of Proposition~\ref{arrows prop}. Since the $1$-localization will be used often here, we simply call it {\em localization of} $\CS$ and we denote it by $\Loc(\CS)$. 
Given a collection of actions $\CS$, we define:
\begin{itemize}
    \item $\CS^\sharp:=\{S\in\CS\colon \text{the transformation in the}\ a_2\text{-position is not}\ I\};$ and
    \item $\CS^\flat = \CS\setminus\CS^\sharp$.
\end{itemize}

The key observation and the reason for considering the localization is that for any collection of actions $\CS$, we have $c_3(\CS) = \ell(\Loc(\CS))$, and for all $k\in\N$, $c_3(\CR_k(\CS))\leq \ell(\Loc(\CR_k(\CS)))$.
Consequently, to find $k\in \N$ such that $c_3(\CR_k(\CS))=1$, it is enough to show $\ell(\Loc(\CR_k(\CS)))=1$.

\begin{Lemma}\label{last technical lemma}
    Let $\CS$ be a nice collection of actions and $H\in\N$ be such that $c_1(\CR_k(\CS))=c_1(\CS)$ and $c_2(\CR_k(\CS))=c_2(\CS)$ for all $k\in[H]$.
    Then for all $k\in[H]$,
    $$
    \Loc(\CR_k(\CS)) = \CR_k(\Loc(\CS))^\sharp
    \quad\text{and}\quad
    \CR_k(\Loc(\CS))^\flat\neq\emptyset.$$
\end{Lemma}

\begin{proof}[Proof of Proposition~\ref{arrows prop}~\ref{hard arrow} assuming Lemma~\ref{last technical lemma}] 
We use induction on $c_1$. The base case $c_1=1$ holds trivially as the statement is void. For the induction step, let $c_1=m>1$ and suppose that the claim holds for $1,2,\ldots,m-1$. 
Let $\CS$ be a nice collection of actions with complexity vector $\mathbf{c}=(m,c_2,c_3)$, $c_3>1$. We want to show that $c_3(\CR_k(\CS))=1$ for some $k\in\N$. 
As already remarked, it is enough to show that $\ell(\Loc(\CR_k(\CS)))=1$ for some $k$. 
In view of Lemma~\ref{last technical lemma}, it suffices to show that $\ell(\CR_k(\Loc(\CS))^\sharp)=1$ for some $k$, provided that $c_1(\CR_{k'}(\CS))=c_1(\CS)$ and $c_2(\CR_{k'}(\CS))=c_2(\CS)$ for all $k'\in[k]$.
We may assume that the latter condition holds; otherwise, in view of Remark~\ref{same c_1,c_2 and transformation}, there would be some $k'<k$ such that $c_3(\CR_{k'}(\CS))=1$, concluding the proof.
For the collection $\CR_k(\Loc(\CS))^\sharp$ to have length $1$, by the same lemma, the collection $\CR_k(\Loc(\CS))$ must consist of actions with identity in the second position and exactly one action with a nonidentity transformation there. 
Thus, to conclude, we shall find $k$ such that
$$
    c_2(\CR_k(\Loc(\CS))) = c_3(\CR_k(\Loc(\CS))) = 1.
$$

Since $c_1(\Loc(\CS))\leq m-1$, we may apply the induction hypothesis to $\Loc(\CS)$ and to its successive reductions. 
We first reduce $c_3$ to $1$, if necessary, by the induction hypothesis, without changing the first two complexity parameters. 
If $c_2>1$, then we apply Proposition~\ref{arrows prop}~\ref{easy arrow 1} to reduce to $c_2$ to $c_2-1$. 
Iterating this, after finitely many steps, we have reduced $\Loc(\CS)$ to a collection with $c_2=1$, and one last application of the induction hypothesis, if necessary, will further reduce to a collection with $c_2=c_3=1$, proving the assertion, thus, concluding the proof.
\end{proof}

It remains to establish Lemma~\ref{last technical lemma}. 

\begin{proof}[Proof of Lemma~\ref{last technical lemma}]
Let $\CS = \{S_i\colon i\in[\ell]\}$ be a nice collection of actions, where for each $i\in[\ell]$, $n\in\N$, 
$$S_{i,n}=\prod_{j=1}^mT_{i,j}^{a_j(n)},$$
with complexity vector $\mathbf{c}=(c_1,c_2,c_3)$, where $c_1=m$, and we may assume that $c_2=c(a_1)$ and $c_3$ is attained by $T_{1,1}:=T$ (i.e., $i_0=1$ in the complexity reduction procedure). Thus, $\Loc(\CS)$ is precisely the subset of $\CS$ consisting of the actions $S_i$ such that $T_{i,1}=T$. By relabeling, we may assume that $\Loc(\CS)=\{S_i\colon i\in[c_3]\}$. Thus,
$$S_{i,n} = T^{a_1(n)}\prod_{j=2}^{m}T_{i,j}^{a_j(n)}$$
for all $i\in[c_3]$.

We prove simultaneously by induction on $k\in[H]$ both the stated conclusion of the lemma and the following auxiliary assertion: 
Denote $\CR=\CR_{k-1}(\CS)$ and $\CL=\CR_{k-1}(\Loc(\CS))$. Consider the active positions $(\wt{j}_d)_{d\in[\wt{D}]}$ and $(j_d')_{d\in[D']}$, the localizations $(\Loc_d(\CR))_{d\in[\wt{D}]}$ and $(\Loc_d(\CL))_{d\in[D']}$, and the pivots $\wt{P}$ and $P'$ associated to the reduction step applied to $\CR$ and $\CL$ respectively. Then $\wt{j}_1=1$, $j_1'=2$. In addition:
\begin{itemize}
    \item If $c_2(\CL) = 1$, then $D'=\wt{D}$, and for all $2\leq d\leq \wt{D}$, we have $j_d'=\wt{j}_d$ and $\Loc_d(\CL)=\Loc_d(\CR)$.
    \item Otherwise, $D'=\wt{D}-1$, and for all $2\leq d\leq \wt{D}$, we have $j_{d-1}'=\wt{j}_d$ and $\Loc_{d-1}(\CL) = \Loc_d(\CR)$.
\end{itemize}
Finally, in both cases, $P'=\wt{P}$.

The claim $\wt{j}_1=1$ follows from Remark~\ref{same c_1,c_2 and transformation}; the rest is proved by induction.

The base case is similar to the induction step, but simpler, since there are no preceding data to inherit and all choices are made afresh. We therefore omit its proof.
For the induction step, let $k\ge2$, suppose that the claim holds for $k-1$, and we will show that it holds for $k$ as well.
By the induction hypothesis,
\begin{equation}\label{eq_red_hypothesis}
    \Loc(\CR)=\CL^\sharp,
    \qquad
    \CL^\flat\neq\emptyset,
\end{equation}
which implies 
\begin{equation}\label{eq_spliting_into_sharp_and_flat}
    \CL = \Loc(\CR)\cup\CL^\flat,
    \qquad
    \CL^\flat\neq\emptyset.
\end{equation}

We first show that $j_1'=2$.
Since $\Loc(\CR)=\CL^\sharp$ , and $\CL^\flat\neq\emptyset$, the $a_2$-position of $\CL$ is nonconstant. 
Moreover, by the induction hypothesis, the first active position in all preceding reduction steps of $\Loc(\CS)$ was $a_2$. 
By Proposition~\ref{complexity reduction procedure thm}~\ref{item3}, the new action obtained by replacing the first active position of the pivot by $I$ preserves all the transformations in all the other positions, while the remaining new actions involve only transformations already appearing in the corresponding positions. 
Therefore, the sets of transformations appearing in the $a_j$-position, for each $j\geq3$, remain unchanged, whereas the number of distinct nonidentity transformations in the $a_2$-position can only decrease. 
Since $c_2(\Loc(\CS))=c(a_2;\Loc(\CS))$, for every $j\geq3$ such that the $a_j$-position of $\CL$ is nonconstant, we have 
$$
    c(a_2;\CL)
    \leq c(a_2;\Loc(\CS))
    \leq c(a_j;\Loc(\CS))
    = c(a_j;\CL).
$$
Hence $c_2(\CL) = c(a_2;\CL)$, and since the $a_1$-position is constant, then $j_1'=2$.

We write $\CR=\{\wt{S}_i\colon i\in[\wt{\ell}]\}$, for some $\wt{\ell}\in\N$, and for simplicity, we denote its complexity vector by $\wt{\mathbf{c}}=(\wt{c}_1,\wt{c}_2,\wt{c}_3)$. 
By Remark~\ref{same c_1,c_2 and transformation}, $\wt{c}_1=c_1=m$ and $\wt{c}_2=c_2$, so we may write 
$$\wt{S}_{i,n} = \prod_{j=1}^m \wt{T}_{i,j}^{a_j(n)},\ i\in[\wt{\ell}],$$
and in view of Proposition~\ref{complexity reduction procedure thm} \ref{item3}, $\{\wt{T}_{i,1}\colon i\in[\wt{\ell}]\}\subseteq \{T_{i,1}\colon i\in[\ell]\}\cup\{I\}$, and for all $j\geq2$, $\{\wt{T}_{i,j}\colon i\in[\wt{\ell}]\}\subseteq\{T_{i,j}\colon i\in[\ell]\}$.
We let $\bar{\ell}=\ell(\Loc(\CR))$, $\ell'=\ell(\CL)$, where $\ell'>\bar{\ell}$, by \eqref{eq_spliting_into_sharp_and_flat}. 
Since $\wt{j}_1=1$, $j_1'=2$, both pivots $\wt{P},P'$ start with $\wt{T}_{1,1}=T$.

Let $\CF=\{F_i\in L^\infty(\mu)\colon |F_i|=1, i\in[\wt{\ell}]\}$. The reduction step for $(\CR,\CF)$ starts with the $a_1$-position. According to the reduction procedure, once this position is eliminated, we employ Lemma~\ref{Katai for actions} to bound a power of $\limsup_{N\to\infty}\|A_N(\CR,\CF)\|_2$ by 
\begin{multline}\label{finishing with a_1-position}
    \E_{\substack{k_{t,u_t}\in[K] \\ 0\leq t\leq r_1,\underline{u}\in\{0,1\}^{r_1+1}}}\limsup_{N\to\infty}\Re\Bigg(\E_{n\in[N]}\int_X \prod_{i=1}^{\bar{\ell}}\prod_{j=2}^m \wt{T}_{i,j}^{a_j(n)} \\
    \prod_{\underline{u}\in\{0,1\}^{r_1+1}}\bigg(T^{k_{0,u_0}}\prod_{t=1}^{r_1} \big(T\wt{T}_{i_t^{(1)},1}^{-1}\big)^{k_{t,u_t}}\bigg)\CC^{|\underline{u}|}F_i\,d\mu\Bigg),
\end{multline}
for all $K\in\N$, where, by a slight abuse of notation, the parameters $r_1$, $(i_t^{(1)})_{t\in[r_1]}$ are associated with $\CR$ (and not with $\CS$). 
Note that the actions in the integral are exactly the ones in $\Loc(\CR)=\CL^\sharp$, without the $a_1$-position.
Now, we define $\CG=\{G_i\colon i\in[\ell']\}$, for $G_i\in L^\infty(\mu)$ of unit modulus, and we run the complexity reduction step for $(\CL,\CG)$, which starts with the $a_2$-position (since $j_1'=2$).
In view of \eqref{eq_spliting_into_sharp_and_flat}, after the first application of Lemma~\ref{Katai for actions}, the actions in the flat part vanish, thus we have bounded $A_N(\CL,\CG)$ by a quantity as in \eqref{finishing with a_1-position}, containing exactly the same actions.
Recall that the pivot of both reductions lie in the $T$-class, i.e., in $\Loc(\CR_{k-1}(\CS))$. 
Now distinguish cases to prove the auxiliary assertion and then conclude in each case separately by comparing $\CR_k(\CS)=\CR_1(\CR)$ and $\CR_k(\Loc(\CS))=\CR_1(\CL)$. 

\textbf{Case 1:} Suppose that $c_2(\CL) = 1$. Then $c(a_2;\CL)=1$. By \eqref{eq_spliting_into_sharp_and_flat}, since the flat part does not contribute, it follows that
$c(a_2;\Loc(\CR)) = c(a_2;\CL)=1$,
and since the $a_2$-position of $\Loc(\CR)$ does not have identities by \eqref{eq_red_hypothesis}, it follows that the second position in this collection is constant.
Since $\Loc(\CS)$ has more than one action, and by the niceness assumption none of its remaining positions contains identities, every nonconstant $a_j$-position, for $j\geq2$, has at least two distinct nonidentity transformations. 
Thus, $c_2(\Loc(\CS))\geq2$. 
Moreover, in each reduction step of $\Loc(\CS)$, the number of distinct nonidentity transformations in the $a_2$-position decreases by at most $1$, since only the transformations of the pivot can disappear from this position. 
Since $c_2(\CL)=c(a_2;\CL)=1$, there is $1\leq k_0\leq k-1$ such that $c_2(\CR_{k-1-k_0}(\Loc(\CS)))=2$. Fix the smallest such $k_0$. In the reduction step taking $\CR_{k-1-k_0}(\Loc(\CS))$ to $\CR_{k-1-k_0+1}(\Loc(\CS))$ the first active position was $a_2$. 
By the induction hypothesis, $a_2$ was the second active position in the reduction step taking $\CR_{k-1-k_0}(\CS)$ to $\CR_{k-1-k_0+1}(\CS)$ but the class of the pivot in that level was eliminated. 
Hence in the following reduction step the second active position was picked according to the selection rule. 
But since the $a_2$-position in $\CR_{k-1-k_0+1}(\Loc(\CS))^\sharp$ is constant and this collection is the same as $\Loc(\CR_{k-1-k_0+1}(\CS))$ by induction hypothesis, the $a_2$-position was skipped in the step taking $\CR_{k-1-k_0+1}(\CS)$ to $\CR_{k-1-k_0+2}(\CS)$, and in all subsequent ones. 
Consequently, $\wt{j}_2>2=j_1'$.

After the first active positions are eliminated, both $A_N(\CR,\CF)$ and $A_N(\CL,\CG)$ are bounded by expressions of the form \eqref{finishing with a_1-position} where the actions are 
$$\prod_{j=3}^m \wt{T}_{i,j}^{a_j(n)},$$
since the $a_2$-position is eliminated because $\wt{T}_{i_1,2}=\wt{T}_{i_2,2}$ for all $i_1,i_2\in[\bar{\ell}]$, and $\wt{j}_2,j_2'>2$.
Let $\wt{d}_{\max}$ and $d_{\max}'$ associated to the reduction step applied to $\CR$ and $\CL$ respectively.
If $k_0=1$, then $\wt{d}_{\max}=1$ and by induction hypothesis, $d_{\max}'=0$, so both $\wt{j}_2,j_2'$ are determined by the selection rule.
Since each active position is determined by the actions of the respective preceding localization, i.e., the actions in the integral in the bound, and these are the same for both reduction steps, it follows that $\wt{j}_2=j_2'$. 
If $k_0>1$, then we shall further distinguish cases. 
If $d_{\max}'=1$, then by the induction hypothesis, $\wt{d}_{\max}=1$, so both $\wt{j}_2,j_2'$ are selected.
Since the actions in both integrals in the bound are the same, the selection rule gives $\wt{j}_2=j_2'$. 
Otherwise, both $\wt{j}_2,j_2'$ are inherited by the previous reduction steps, where the second active positions agree by the induction hypothesis. Thus, in any case, we have $\wt{j}_2=j_2'$.
After the second active position is eliminated, still the actions in the integral in the bounds are the same for both steps, so recursively, by the same argument, we get that $\wt{D}=D'$ and $\wt{j}_d=j_d'$ for all $d\geq2$.
The same argument shows that $\Loc_d(\CL)=\Loc_d(\CR)$ for all $d\geq2$, since at the levels where the data are inherited, this follows by the induction hypothesis, while when a new selection must be done, then the selection for both reduction steps is the same since the actions in the corresponding integral in the bound are the same.
In particular, the two procedures have the same pivot, that is, $P'=\wt{P}$.
This establishes the auxiliary assertion in this case.

By Proposition~\ref{complexity reduction procedure thm} \ref{item2} and by the auxiliary claim, the seminorm bounds after the one reduction step for $(\CR,\CF)$ and $(\CL,\CG)$ are 
$$
\nnorm{F_1}_{T,T\wt{T}_{i_1^{(1)},1}^{-1},\ldots,T\wt{T}_{i_{r_1}^{(1)},1}^{-1},\wt{T}_{1,\wt{j}_2}\wt{T}_{i_1^{(2)},\wt{j}_2}^{-1},\ldots,\wt{T}_{1,\wt{j}_2}\wt{T}_{i_{r_2}^{(2)},\wt{j}_2}^{-1},\ldots,\wt{T}_{1,\wt{j}_{\wt{D}}}\wt{T}_{i_1^{(\wt{D})},\wt{j}_{\wt{D}}}^{-1},\ldots,\wt{T}_{1,\wt{j}_{\wt{D}}}\wt{T}_{i_{r_{\wt{D}}}^{(\wt{D})},\wt{j}_{\wt{D}}}^{-1}},
$$
and 
$$
\nnorm{G_1}_{\wt{T}_{1,2},\wt{T}_{1,\wt{j}_2}\wt{T}_{i_1^{(2)},\wt{j}_2}^{-1},\ldots,\wt{T}_{1,\wt{j}_2}\wt{T}_{i_{r_2}^{(2)},\wt{j}_2}^{-1},\ldots,\wt{T}_{1,\wt{j}_{\wt{D}}}\wt{T}_{i_1^{(\wt{D})},\wt{j}_{\wt{D}}}^{-1},\ldots,\wt{T}_{1,\wt{j}_{\wt{D}}}\wt{T}_{i_{r_{\wt{D}}}^{(\wt{D})},\wt{j}_{\wt{D}}}^{-1}}
$$
respectively, 
where, as before, the parameters $(r_d)_{d\in[\wt{D}]}$, and $(i_t^{(d)})_{0\leq t\leq r_d,d\in[\wt{D}]}$ are associated with $\CR$ (and not with $\CS$). Then, by Proposition~\ref{complexity reduction procedure thm} \ref{item3}, we have
\begin{multline*}
    \CR_k(\CS) 
    = 
    \CR_1(\CR)
    =
    \{\wt{S}_i\colon 2\leq i\leq \wt{\ell}\}
    \cup
    \{(\wt{S}_{1,n}T^{-a_1(n)})_{n\in\N}\} \\ 
    \cup 
    \Bigg\{\Bigg(\wt{T}_{i_t^{(d)},\wt{j}_d}^{a_{\wt{j}_d}(n)}\prod_{\substack{j=1 \\ j\neq \wt{j}_d}}^m \wt{T}_{1,j}^{a_j(n)}\Bigg)_{n\in\N}\colon d\in[\wt{D}], t\in[r_d]\Bigg\},
\end{multline*}
while
\begin{multline*}
    \CR_k(\Loc(\CS)) 
    = 
    \CR_1(\CL)
    =
    \{\wt{S}_i\colon 2\leq i\leq\bar{\ell}\}
    \cup
    \{(\wt{S}_{1,n}\wt{T}_{1,2}^{-a_2(n)})_{n\in\N}\} \\
    \cup 
    \Bigg\{\Bigg(\wt{T}_{i_t^{(d)},\wt{j}_d}^{a_{\wt{j}_d}(n)}\prod_{\substack{j=1 \\ j\neq \wt{j}_d}}^m \wt{T}_{1,j}^{a_j(n)}\Bigg)_{n\in\N}\colon 2\leq d\leq \wt{D}, t\in[r_d]\Bigg\}\cup\CR_{k-1}(\Loc(\CS))^\flat.
\end{multline*}
Therefore, recalling that $1\leq i_t^{(d)}\leq \bar{\ell}$, for all $2\leq d\leq \wt{D}$ and all $0\leq t\leq r_d$, we have
\begin{multline*}
    \Loc(\CR_k(\CS)) 
    = \{\wt{S}_i\colon 2\leq i\leq \bar{\ell}\}
    \cup
    \Bigg\{\Bigg(\wt{T}_{i_t^{(d)},\wt{j}_d}^{a_{\wt{j}_d}(n)}\prod_{\substack{j=1 \\ j\neq \wt{j}_d}}^m \wt{T}_{1,j}^{a_j(n)}\Bigg)_{n\in\N}\colon 2\leq d\leq \wt{D}, t\in[r_d]\Bigg\} \\
    =
    \CR_k(\Loc(\CS))^\sharp,
\end{multline*}
and
\begin{equation*}
    \CR_k(\Loc(\CS))^\flat = \{(\wt{S}_{1,n}\wt{T}_{1,2}^{-a_2(n)})_{n\in\N}\}\cup\CR_{k-1}(\Loc(\CS))^\flat \neq \emptyset.
\end{equation*}
The result follows in this case.

\textbf{Case 2:} Suppose that $c_2(\CL)>1$. 
As before, it follows by \eqref{eq_spliting_into_sharp_and_flat} that 
$$c(a_2;\Loc(\CR)) = c(a_2;\CL)=c_2(\CL),$$
so the $a_2$-position of $\Loc(\CR)$ has at least two distinct nonidentity transformations. 
As before, either by inheritance where the induction hypothesis applies, or by the selection rule, it follows that $\wt{j}_2=2=j_1'$. 
Since the nonidentity transformation classes in the $a_2$-position of $\CL$ are precisely those of $\Loc(\CR)$, the two procedures select the same class at these corresponding levels, so $\Loc(\CL) = \Loc_2(\CR)$. The remaining assertions of the auxiliary claim follow as in Case 1.

By Proposition~\ref{complexity reduction procedure thm}~\ref{item2} and by the auxiliary claim, the seminorm bounds after the one reduction step for $(\CR,\CF)$ and $(\CL,\CG)$ are 
$$
\nnorm{F_1}_{T,T\wt{T}_{i_1^{(1)},1}^{-1},\ldots,T\wt{T}_{i_{r_1}^{(1)},1}^{-1},\wt{T}_{1,\wt{j}_2}\wt{T}_{i_1^{(2)},\wt{j}_2}^{-1},\ldots,\wt{T}_{1,\wt{j}_2}\wt{T}_{i_{r_2}^{(2)},\wt{j}_2}^{-1},\ldots,\wt{T}_{1,\wt{j}_{\wt{D}}}\wt{T}_{i_1^{(\wt{D})},\wt{j}_{\wt{D}}}^{-1},\ldots,\wt{T}_{1,\wt{j}_{\wt{D}}}\wt{T}_{i_{r_{\wt{D}}}^{(\wt{D})},\wt{j}_{\wt{D}}}^{-1}},
$$
and 
$$
\nnorm{G_1}_{\wt{T}_{1,\wt{j}_2},\wt{T}_{1,\wt{j}_2}\wt{T}_{i_1^{(2)},\wt{j}_2}^{-1},\ldots,\wt{T}_{1,\wt{j}_2}\wt{T}_{i_{r_2}^{(2)},\wt{j}_2}^{-1},\ldots,\wt{T}_{1,\wt{j}_{\wt{D}}}\wt{T}_{i_1^{(\wt{D})},\wt{j}_{\wt{D}}}^{-1},\ldots,\wt{T}_{1,\wt{j}_{\wt{D}}}\wt{T}_{i_{r_{\wt{D}}}^{(\wt{D})},\wt{j}_{\wt{D}}}^{-1}}
$$
respectively, with $\wt{j}_2=2$. Now, exactly as in the previous case, by Proposition~\ref{complexity reduction procedure thm}~\ref{item3}, we write down $\CR_k(\CS)$ and $\CR_k(\Loc(\CS))$ explicitly, and conclude by comparing $\Loc(\CR_k(\CS))$ and $\CR_k(\Loc(\CS))$. This concludes the induction step, thus, the proof of the lemma.
\end{proof}

\bibliographystyle{abbrv}
\bibliography{refs}

\end{document}